\documentclass[11pt]{amsart}

\usepackage{graphicx,amssymb,amsfonts,latexsym,amsmath,amsthm,times, amscd, amsxtra, euscript, MnSymbol}
\usepackage{pdfsync}
\usepackage{enumerate}
\usepackage{fullpage}
\usepackage{float}
\newtheorem{theorem}{Theorem}[section]
\newtheorem{lemma}[theorem]{Lemma}
\newtheorem{proposition}[theorem]{Proposition}
\newtheorem{corollary}[theorem]{Corollary}
\theoremstyle{definition}

\newtheorem{example}[theorem]{Example}

\newtheorem{problem}[theorem]{{\sc Problem}}
\newtheorem{problems}[theorem]{{\sc Problems}}

\newtheorem{claim}{\noindent {\bf Claim}}

\newcommand{\R}{\mathbb R}
\newcommand{\N}{\mathbb N}

\newcommand{\Z}{\mathbb Z}
\newcommand{\Q}{\mathbb Q}

\DeclareMathOperator{\Cong}{Cong}
\DeclareMathOperator{\Eqv}{Equiv}
\DeclareMathOperator{\Inv}{Inv}
\DeclareMathOperator{\Pol}{Pol}

\title[Preservation of discrete structures] {Preservation of discrete structures. A metric point of view}
\author[C.Delhomm\'e]{Christian Delhomm\'e}
\address{LIM-ERMIT, D\'epartement de Math\'ematiques et Informatique, Universit\'e de La R\'eunion, 15 avenue Ren\'e Cassin - BP 7151 - 97715 Saint-Denis Messag. Cedex 9 FRANCE}
\email[Christian Delhomm\'e]{delhomme@univ-reunion.fr}
\author[M.Miyakawa]{Masahiro Miyakawa}
\address{Tsukuba University of Technology,
4-12-7 Kasuga, Tsukuba,
 Ibaraki 305-8521, Japan} \email {mamiyaka@cs.k.tsukuba-tech.ac.jp, tatsumi@cs.k.tsukuba-tech.ac.jp}
\author[M.Pouzet] {Maurice Pouzet}
\address{Univ. Lyon, Universit\'e Claude-Bernard  Lyon1, CNRS UMR 5208, Institut Camille Jordan, 
43, Bd. du 11 Novembre 1918,
69622 Villeurbanne, France et Department of Mathematics and Statistics, The University of Calgary, Calgary, Alberta, Canada}
\email{pouzet@univ-lyon1.fr}
\author [H. Tatsumi]{Hisayuki Tatsumi}
\address{Tsukuba University of Technology,
4-12-7 Kasuga, Tsukuba,
Ibaraki 305-8521, Japan} \email{
tatsumi@cs.k.tsukuba-tech.ac.jp}
\date{\today}


\begin{document}

\dedicatory {\dagger \; Dedicated to the memory of Ivo G.Rosenberg}

\begin{abstract} In the early 80's,  Alain Quilliot presented an approach of  ordered sets and graphs in terms  of metric spaces, where instead  of positive real numbers, the values of  the distance are elements of an ordered monoid equipped with an involution. This point of view was further developed in a series of papers by Jawhari, Misane, Pouzet, Rosenberg and Kabil  \cite {jawhari-al, pouzet-rosenberg, kabil-pouzet}. 
Some  results are currently published by Rosenberg, Kabil and Pouzet, Bandelt, Pouzet and Sa\"idane, Khamsi and Pouzet
\cite {KPR, bandelt-pouzet, 
bandelt-pouzet-saidane,
khamsi-pouzet}. Special aspects were developed by the authors of the present paper  \cite{MPRT, delhomme-pouzet}.  A survey  on generalized metric spaces is in print \cite{kabil-pouzet2}. In this paper, we review briefly the salient aspects of the theory of generalized metric spaces, then we  illustrate the properties of the preservation,  by operations,  of sets  of relations,  notably binary relations, and particularly equivalence relations. 

\end{abstract}
\maketitle

\section{Presentation}
We  survey  some aspects of the preservations by operations of discrete structures like graphs, ordered sets and transition systems. In those aspects, the influence of Ivo Rosenberg was prominent to say the less. Despite the sadness of his passing, we are pleased to give hommage to his influence. 

We consider graphs, ordered sets and transition systems  as generalized metric spaces. Instead of positive real numbers, the values of  the distance are elements of an ordered monoid equipped with an involution; the  notion  of non-expansive map particularly fit in this frame. Several notions and results about ordinary metric spaces and their non-expansives maps extend to these structures, e.g., the characterization of injective objets by Aronszajn and Panitchpakdi \cite{aronszajn-panitchpakdi}, as well as the fixed point result of Sine and Soardi \cite{sine, soardi} (see Espinola-Khamsi \cite{espinola-khamsi} for a survey of classical results on metric spaces). As an illustration,   we mention that in the category of oriented reflexive graphs with Quilliot's zigzag distance, the absolute retracts are the retracts of products of zigzags (see \cite{bandelt-pouzet, bandelt-pouzet-saidane}) and, on those which are bounded, any set of graph homomorphism  which commute pairwise has  a fixed point (\cite{khamsi-pouzet}).  Also, as a byproduct  of the description of the injective envelope of  a two-element transition system, Kabil, Rosenberg and the third author of this paper \cite{KPR} have shown that 
the monoid of the final sections of the free monoid  of words on a finite alphabet is also free.

 The approach of graphs and ordered sets  as metric spaces  has its origin in Alain Quilliot's work done in the early 80's  \cite{Qu1, Qu2}. It was  developed  by Jawhari, Misane and the third  author of this paper \cite {jawhari-al}. Later on, Rosenberg and the third  author of this paper \cite{pouzet-rosenberg} developped a more general approach,  seeing operations  preservating  relations as kind of non-expansive maps. They put some emphasis  
 on the study of operations preserving systems of equivalence relations on  a set, viewing these systems as ultrametric spaces. More recently, this lead to the study of semirigid structures (for which the only unary self maps are the identity and the constants), notably those made of systems of three equivalence relations \cite{MPRT, delhomme-pouzet}. 
 
 This paper is composed as follows. 
We introduce directed graphs equipped with the zigzag distance as a motivating example. 
Then,  we  present briefly a survey of properties of generalized metric spaces over a Heyting algebra, alias a dual integral involutive quantale (in short a $D^2I$-quantale). 
We illustrate the main results with ordinary metric spaces, ultrametric spaces, graphs (directed or oriented), ordered sets and transition systems. We use the notion of injective envelope to prove the freeness of the algebra of nonempty final segments of the set of words on an alphabet, a recent result of Kabil, Rosenberg and the third author of this paper \cite{KPR}. We defer the proofs to \cite {jawhari-al} and the forthcoming survey \cite{kabil-pouzet}. We introduce to the duality between relations and operations, via the notion of preservation. We focuse on the case of binary relations and more specifically on equivalence relations. Sublattices of the lattice of equivalence relations on a set play an important role in algebra,  basic examples being congruence lattices of   algebras.  Among those are  arithmetical lattices. They lead to generalized ultrametric spaces close to the  hyperconvex ones (they are finitely hyperconvex). A typical arithmetical lattice is the lattice of congruences of the additive group $\Z$. We present the result of C\'egielski, Grigorieff and Guessarian (CGG), 2014 \cite{cgg1, cgg2} describing the operations preserving this lattice; we give a short proof of the crucial argument due to the first author of this paper. The case of the additive group of $\Z\times \Z$ is completely different. While on the lattice of congruences of $\Z$  the number of operations preserving the congruences is the continuum, there are only countably many operations preserving the congruences of $\Z\times \Z$. In fact preserving three congruences to preserve all is enough. Next,  we go the the opposite direction: the study of semirigidity. A relational structure is \emph{semirigid} if the only operations preserving it are the constant and projection maps. In the case of  a set of equivalence  relations this amount to the fact that the self maps preserving it are the constant maps and the identity map. We recall Zadori's result proving that for $n\not = 4, 2$ there is a semirigid set of  three equivalence relations  an $n$-element set.  Using geometric properties of the plane, we show that for each cardinal $\kappa$, $\kappa\not \in \{2,4\}$ and  $\kappa\leq 2^{\aleph_0}$, there exists a semirigid system  of three equivalences on a set of cardinality $\kappa$. We leave the question for $\kappa> 2^{\aleph_0}$ as a conjecture. 

\section*{acknowledgement} This paper was prepared while  the third author of this paper stayed at the University of Tsukuba from May 8 to July 6, 2019. Support provided by  the  JSP is gratefully acknowledged. Some of the results were presented at the Department of Electronics and Computer  Engineering   of the  Hiroshima Institute of Technology,   seminar of professor Tomoyuki Araki (June 6, 2019), at the Department of Intelligent robotics, Toyama Prefectural University, seminar of professor Noboru Tagaki, (June 12, 2019) and at the Dept. of Mathematics of the University of Kyoto, seminar of  professor Taketomo Mitsui (June 26 2019); professors Araki, Tagaki and Mitsui are warmly thanked for their welcome.The third author of this paper is pleased to thank Serge Grigorieff for discussion about CGG's theorem and bibliographical informations.

\section{Motivation: distances on graphs}
We present the zigzag distance on directed graphs introduced by Alain Quilliot in 1983 \cite{Qu1, Qu2}. 

Let $G$ be an undirected graph, that is,  a pair $(V, \mathcal E)$, where $V$ is a set of elements called  \emph{vertices}\; and  $\mathcal E$ is a set of pairs of (distinct) vertices called \emph{edges}. If $x$ and $y$ are two vertices, the \emph{distance}\;  from $x$ to $y$, denoted by $d_G(x,y)$,  is $0$ if $x=y$,  $m$ if $m$ is the length of the shortest path joining $x$ to $y$ and $\infty$ if there is no path between $x$ and $y$ (that is $x$ and $y$ belong to two different connected components of $G$).  

This notion of distance with integer values plays a basic role in the study of graphs. 
In \cite{Qu1, Qu2}, Alain Quilliot proposed an  extension of  this notion to ordered sets and directed graphs.

A \emph{{directed graph}}\;  $G$ is a pair $(V, \mathcal E)$, where $\mathcal E$ is a binary relation on $V$. We say that  $G$ is  \emph{reflexive}\;  if $\mathcal E$ is reflexive and we say  that $G$ is \emph{oriented}\;  if $\mathcal E$ is \emph{antisymmetric} (that is $(x,y)$ and $(y,x)$ cannot be in $\mathcal E$ simultaneously except if $x =y$). Typical examples of oriented graphs are \emph{{ordered sets}}.

If $\mathcal E$ is symmetric, and reflexive, we may identify it with a subset of pairs of distinct elements of $V$, so we get an undirected graph. 
\subsection{From paths to zigzags}

 Let us recall that a finite \emph{path}\;  is an undirected graph $L:=(V, \mathcal E)$ such that one can enumerate the vertices into a non-repeating sequence  $v_0, \dots, v_n$ such that edges are the  pairs $\{v_i,v_{i+1}\}$ for $i<n$. 
 
 Recall that if  $G:= (V, \mathcal E)$ and $G':= (V', \mathcal E')$ are two  graphs,  a \emph{{homomorphism}} \;  from $G$ to $G'$ is a map $h:V\rightarrow   V'$ such that $(h(x), h(y)) \in \mathcal E'$ whenever $(x,y)\in \mathcal E$ for every $(x,y)\in V \times V$. If $G'$ has loops, a  homomorphism can collapse an edge, otherwise  it cannot. We  suppose  that our graphs are reflexive, i.e.,  have  a loop at every vertex, this simplifies considerably the study.
 
 With these two definitions, the distance from a vertex $x$ to a vertex  $y$ in an undirected  graph $G$ is the least integer $n$  such that there an homomorphism from a path $L$ of length $n$  into $G$ and sending an extremity of $L$ on $x$ and the other on $y$, otherwise it is infinite.  
 
 This definition extends to  directed graphs as follows.

\subsection{Zigzags}
 
  A {\emph{reflexive zigzag}}\;  is a reflexive graph $Z$ such that the symmetric hull is a path.

\begin{figure}[h]
\centering
\includegraphics[width=0.3\textwidth]{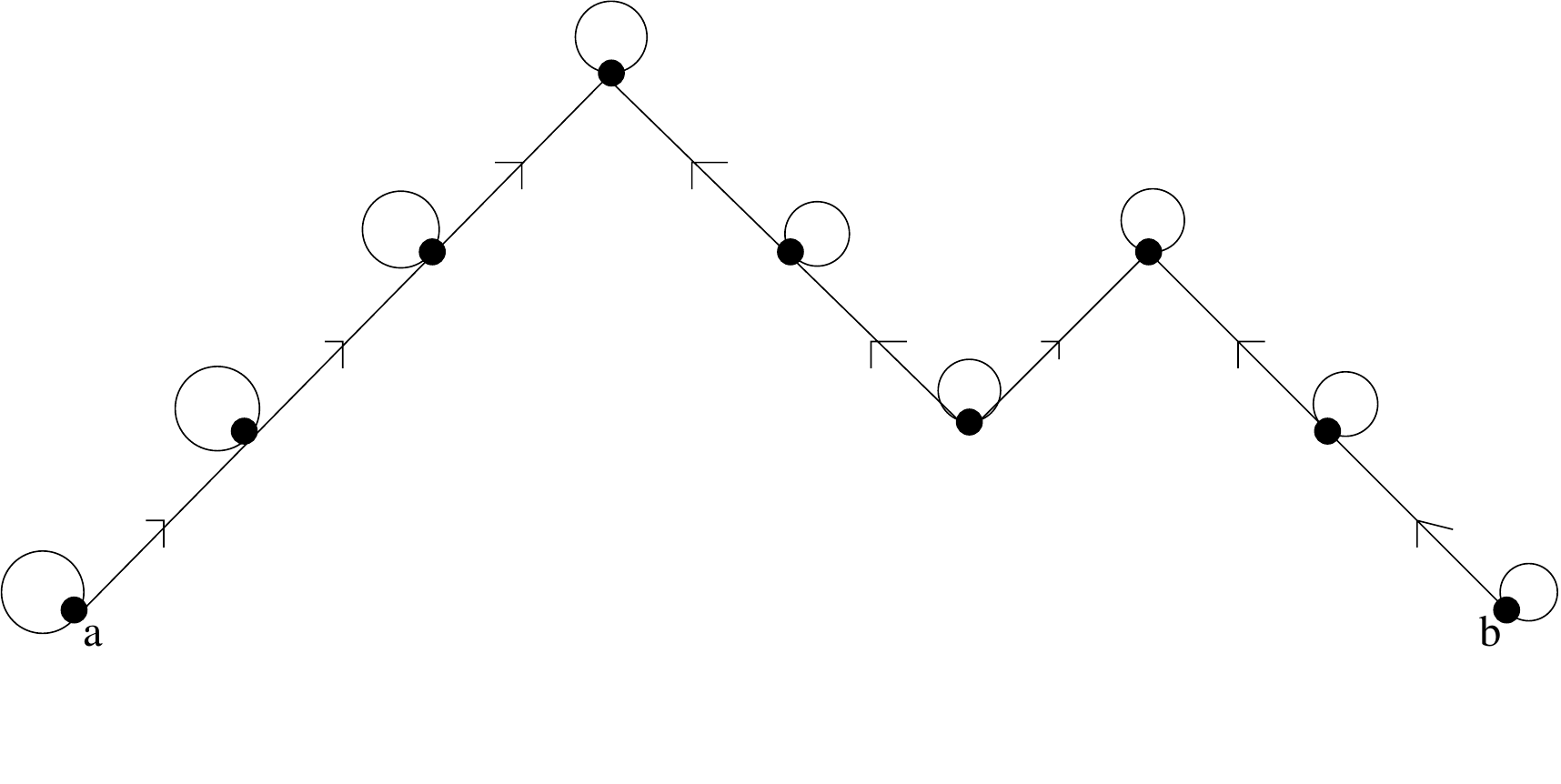}
\caption{A reflexive oriented zigzag}
\label{orientzigzag}
\end{figure}

\begin{figure}[h]
\centering
\includegraphics[width=0.3\textwidth]{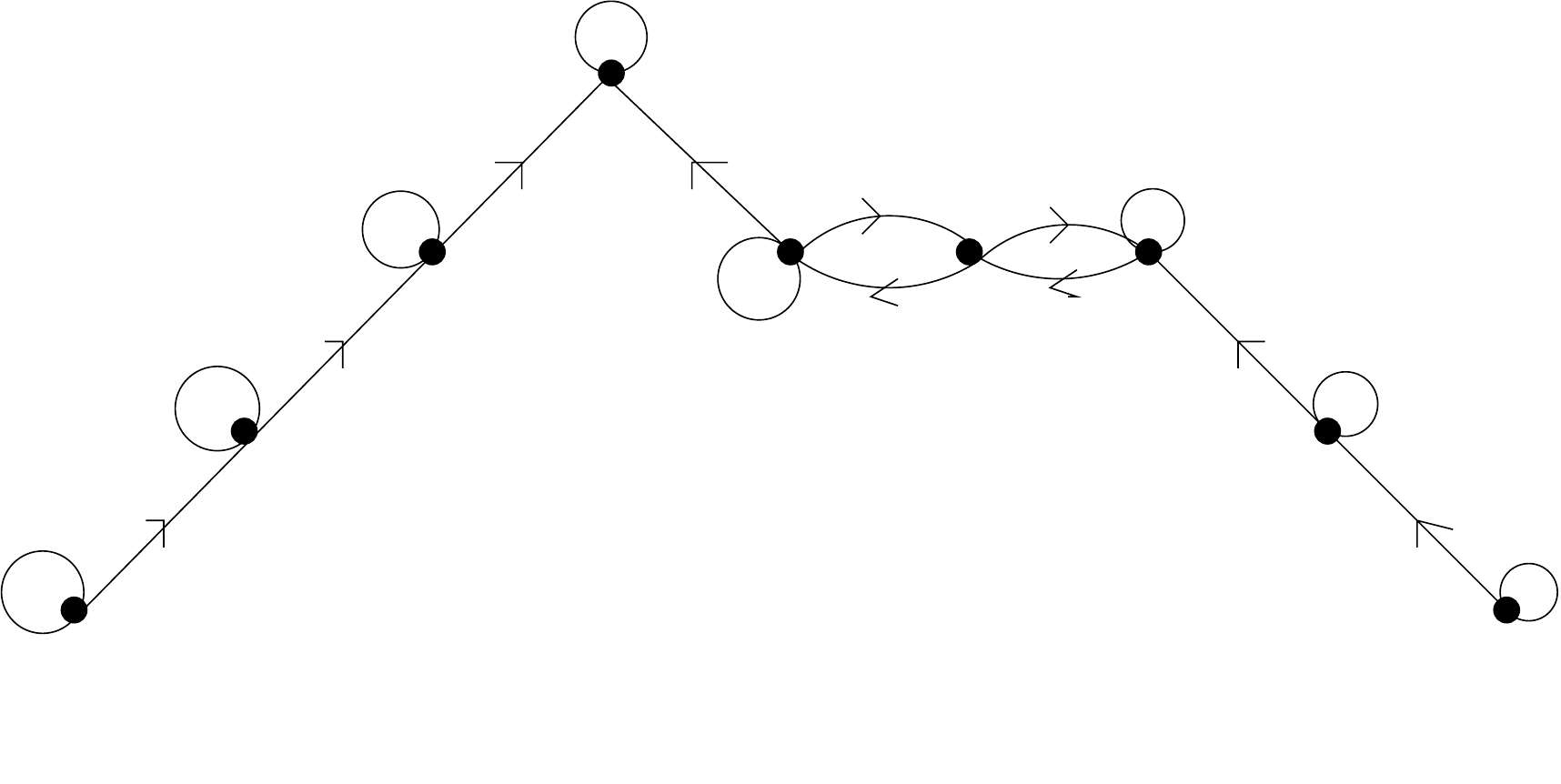}
\caption{A reflexive directed zigzag}
\label{directzigzag}
\end{figure}

 If $x$ and $y$ are two vertices of a graph $G$, we may look at all the zigzags $Z$ which can be mapped into $G$ by a homomorphism sending the extremities  of $Z$ on $x$ and $y$ respectively.   
 
  \begin{figure}[h]
\centering
\includegraphics[width=0.8\textwidth]{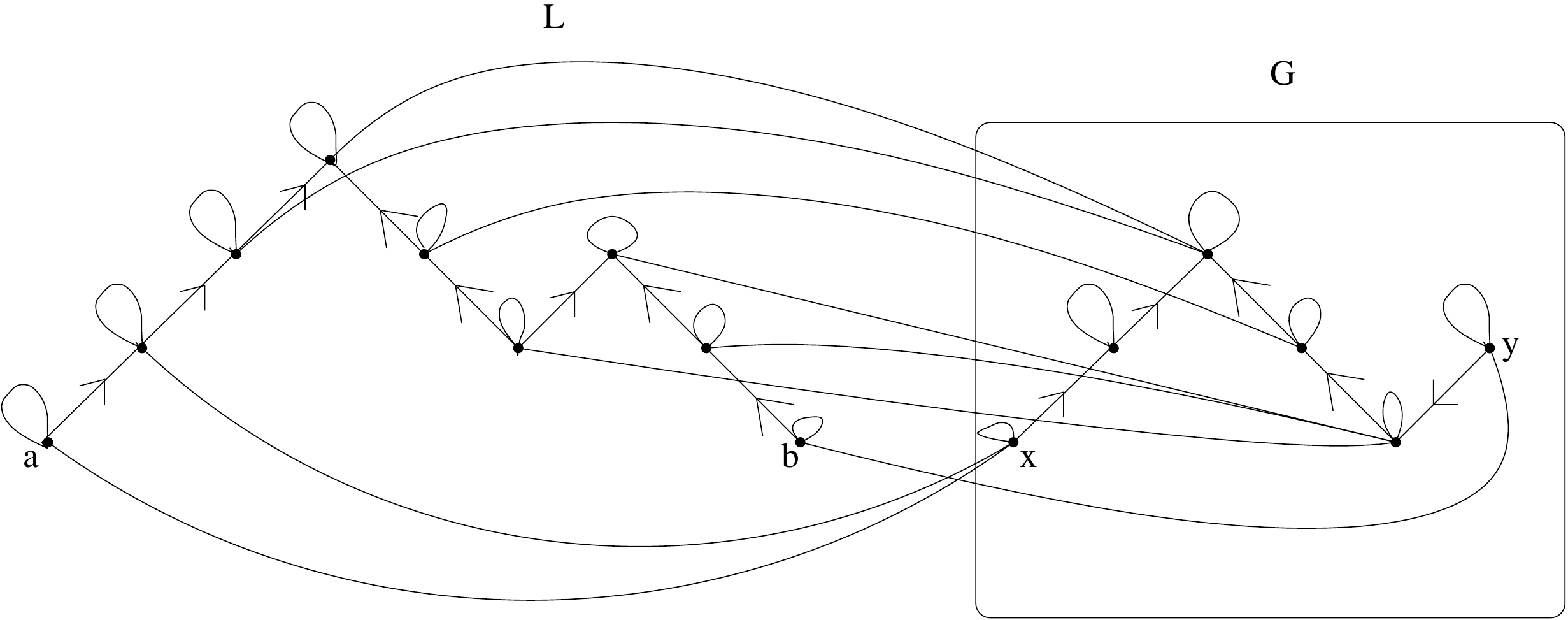}
\caption{A morphism of an oriented zigzag into a directed graph}
\label{figure-B-P-S-2019}
\end{figure}

  In order to simplify, we consider only reflexive oriented zigzag. 
  
  If  $L$ is a reflexive oriented zigzag, we may 
 enumerate the vertices in a non-repeating sequence $v_0:= x, \dots, v_i, \dots,  v_{n}:= y$ 
 and to this enumeration we may  associate the word $ev(L):= \alpha_0\cdots \alpha_i \cdots\alpha_{n-1}$ over the alphabet $\Lambda:= \{+,-\}$, where $\alpha_i:= +$ if $(v_i,v_{i+1})$ is an edge and $\alpha_i:= -$ if $(v_{i+1},v_{i})$ is an edge. If the path has just one vertex, the word will be the empy set and will be denoted by $\Box$.  Conversely, to a  word $u:= \alpha_0\cdots \alpha_i \cdots\alpha_{n-1}$ we may associate the reflexive zigzag $L_u:= (\{0, \dots n\}, \delta_u)$ with extremities $0$ and $n$ (where $n$ is the length $\ell(u)$ of $u$) such that $\delta_u(i,i+1):=u_i $ for $i<n$. 
\subsection{The zigzag distance}
  
 Let $G:= (V, \mathcal E)$ be a reflexive directed graph. For each pair $(x,y)\in V\times V$, the \emph{\emph{zigzag distance}}\;  from $x$ to $y$ is the set $d_G(x,y)$ of words $u$ such that there is graph homomorphism $h$ from $L_u$ into $G$ which sends $0$ on $x$ and $\ell(u)$ on $y$.
 
 This notion is due to Quilliot (Quilliot considered reflexive directed graphs, not necessarily oriented, and in defining the distance, considered only  oriented paths). A general study is presented in Jawhari-Misane-Pouzet 1986 \cite{jawhari-al} ; some developments appear in Pouzet-Rosenberg  1994 \cite{pouzet-rosenberg}  and in Kabil-Pouzet 1998 \cite{kabil-pouzet}.

Let us see that this   map has the properties of a distance.

\subsection{The set of values}
 The set $d_G(x,y)$ is a subset of 
$\Lambda^*$,   the set  of words over the alphabet  $\Lambda:= \{+,-\}$. Because of the reflexivity of $G$, every word obtained from a word belonging to $d_G(x,y)$ by inserting letters  will be also into $d_G(x,y)$. That is, $d_G(x,y)\in \mathbf F(\Lambda^*)$,  the set of final segments of $\Lambda^*$ ordered by the subword ordering.  To see that the map $d_G: V\times V\rightarrow  \mathbf F(\Lambda^*)$ has properties similar to a distance, let us put in evidence some properties of $\mathbf F(\Lambda^*)$.

Extend the involution on $\Lambda$ exchanging  $+$ and $-$ to $\Lambda^*$ by setting $\overline \Box:= \Box$ and $\overline {u_0\cdots u_{n-1}}:= \overline{u_{n-1}}\cdots \overline{u_{0}}$ for every word in $\Lambda^*$. Next,  set $\overline X:= \{\overline u: u\in X\}$ for any set $X$ of words and note that  $\overline X$ belongs to  $\mathbf F(\Lambda^*)$, whenever $X$ belongs to it. 
Order $\mathbf F(\Lambda^*)$  by reverse of the inclusion, denote by   $0$ its least element (that is $\Lambda^*$), set $X\oplus Y$ for $X\cdot Y:= \{uv: u\in X, v\in Y\}$. 

\subsection{The distance}

The map $d_G$ satisfies the following properties:

 \begin{enumerate}[(i)]
 
 \item $d_G(x,y)  = 0$ iff $x=y$;
 
 \item $d_G(x,y) \leq d_G(x,z)\oplus d_G(z, y)$;
 
 \item $\overline{d_G(y,x)}=d_G(x,y)$.
 
 \end{enumerate}

Several categorical properties of (reflexive) graphs and their homomorphisms depend upon some properties of  $\mathbf F(\Lambda^*)$. 

A crucial one is that this set is an ordered monoid equipped with an involution and satisfying the following  distibutivity condition:

\begin{equation}\label{distribinfty}
\bigcup  _{\alpha \in A, \beta \in B} U_\alpha \oplus  V_\beta =
\bigcup _{\alpha \in A} U_\alpha  \oplus \bigcup _{\beta \in B} V_\beta
\end{equation}
for all $U_\alpha \in \mathbf F(\Lambda^*)$   $(\alpha \in A)$ and
$V_{\beta} \in \mathbf F(\Lambda^*)$   $(\beta \in B)$.

As this will be shown below,  it turns out from this  distributivity condition that the distance set $\mathbf F(\Lambda^*)$ can be equipped with  a graph structure, say $G_{\mathbf F(\lambda^*)}$;     the zigzag distance $d_{\mathbf F(\Lambda^*)}$ associated with this graph is such that every metric space over $\mathbf F(\Lambda^*)$ can be embedded isometrically  into a power of the space $(\mathbf F(\Lambda^*), d_{\mathbf F(\Lambda^*)})$, furthermore, every (directed) graph can be isometrically embedded into a power of $G_{\mathbf F(\lambda^*)}$. 

These facts hold for generalized metric spaces over any ordered monoid satisfying this distributivity condition.  In 
 the next section,  we present  shortly a general frame  and review the most salient facts. For the proofs see \cite {jawhari-al} and, for more details, see the forthcoming survey \cite{kabil-pouzet2}.

\section{Generalized metric space}\label{section:generalized metric spaces}
 Generalizations of the notion of a metric space  are  as old as the notion of ordinary metric space  and arise  from geometry and logic, as well as probability (see \cite{blumenthal}, \cite{blumenthal1} \cite{blumenthal-menger}, \cite{lawvere}). The generalization we consider here, originating  in  Jawhari et al \cite{jawhari-al},  is one among several. The basic object,  called a Heyting algebra,  is an  ordered monoid equipped with an involution satisfying a distributivity condition. It was pointed out recently to the third author that the study of these Heyting algebras goes back to the late 1930's and the work of M. Ward and R. P. Dilworth (1939) \cite{ward-dilworth} and also that a more appropriated term  would have been   a \emph{dual integrale involutive quantale} (in short a $D^2I$-\emph{quantale}) see \cite{kaarli-radeleczki}. The term quantale was introduced in $1984$ by C. J. Mulvey \cite{mulvey} as a combination of "quantum logic'' and "locale'', see Rosenthal \cite {rosenthal} and the recent book of Eklund et al \cite{Eklund-al}) about quantales.

\subsection{The objects}
Let $\mathcal H$ be an ordered monoid equipped with an involution. We denote by $\oplus$ the monoid operation, by $0$ its neutral element  and by $^{-}$ the involution, so that $\overline {p+q}= \overline{q} \oplus \overline{q}$ for all $p,q\in \mathcal H$. 

 From now on, we suppose that the neutral element of the monoid $\mathcal H$ is the least element of $\mathcal H$ for the ordering.

Following
Pouzet-Rosenberg, 1994 \cite{pouzet-rosenberg}, we say that  a set $E$ equipped with a map $d$ from $E\times E$ into $\mathcal H$ and which satisfies   properties $(i), (ii), (iii)$     stated below  is a $\mathcal H$-\emph{distance}, and the pair $\bold E:= (E,d)$   is a  $\mathcal H$-\emph{metric space}. 

\begin{enumerate}[(i)]
\item $d(x,y)  = 0$ iff $x=y$;
\item $d(x,y) \leq d(x,z)\oplus d(z, y)$;
\item $\overline{d(y,x)}=d(x,y)$.
\end{enumerate}

In  the Encyclopedia of distances  \cite{deza-deza} (cf. p.82) the corresponding $\mathcal H$-metric spaces are called   \emph{generalized distance spaces}\;  and the maps $d$  are called  \emph{generalized metrics}.

\subsection{The morphisms}

If  $\mathbf E:=(E, d)$ is a $\mathcal H$-metric space and $A$ a subset of $E$, the restriction of $d$ to $A\times A$, denoted by $d_{\restriction A}$,  is a $\mathcal H$-distance
and $\mathbf A:= (A, d_{\restriction A})$ is a \emph{{restriction}}\; of $\mathbf E$. As in the case of ordinary metric spaces, if  $\mathbf E:=(E,d)$ and $\mathbf E':=( E', d')$ are two
$\mathcal H-$metric spaces,  a map $f:E\longrightarrow E^{\prime }$ is a
\emph{non-expansive map}\;  (or a \emph{contraction}) from $\mathbf E$ to $\mathbf E'$ provided that $d^{\prime }(f(x),f(y))\leq d(x,y)$ holds for all $%
x,y\in E$. The map $f$ is an \emph{{isometry}}\;  if $d^{\prime
}(f(x),f(y))=d(x,y)$ for all $x,y\in E$.

\subsection{Retracts and fix-point property}
The space $\mathbf E:=(E,d)$ is a \emph{{retract}}\;  of $\mathbf E':=(E',d')$, if there are two non-expansive maps $f:
\mathbf E\longrightarrow \mathbf E'$ and $g:\mathbf E'
\longrightarrow \mathbf E$ such that $g\circ f=id_{\mathbf E}$ (where $id_{\mathbf E}$ is
the identity map on $\mathbf E$). In this case,
$f$ is a  \emph{{coretraction}}\;  and $g$ a \emph{retraction}. If $\mathbf E$ is a subspace of $\mathbf E'$, then clearly $\mathbf E$ is a retract of
$\mathbf E'$ if there is a non-expansive map from $\mathbf E'$ to $\mathbf E$ such $%
g(x)=x $ for all $x\in E$. We can easily see that every coretraction is
an isometry. 
A generalized metric space $\mathbf E$ is an \emph{absolute retract} w.r.t. isometries   if it is a
retract of every isometric extension. This notion is related to the fixed-point property.
\begin{claim}
If a metric space  has the fixed-point property (f.p.p.), i.e.,  
every non-expansive self map has a fixed point, then every retract has it. 
\end{claim}

Consequence: if there are plently of metric spaces with the f.p.p. then absolute retracts are good candidates.

\subsection{Injectivity, extension property, hyperconvexity} We introduce three other notions.  
Let $\mathbf E$ be a  metric space. This space  is said to be \emph{injective}   if for all spaces $\mathbf E'$ and $\mathbf E''$,
each non-expansive map $f : \mathbf E' \to \mathbf E$, and every isometry
$g : \mathbf E' \to \mathbf E''$ there is a non-expansive map $ h : \mathbf E'' \to \mathbf E$ such that
$ h\circ g =f$.\\ 
 The metric space $\mathbf E$ has the \emph{extension property}\;  if for every space $\mathbf E'$, every non-expansive map  $f$ from a subset $A$ of $\mathbf E'$ into $\mathbf E$ extends to a non-expansive map from $\mathbf E'$ to $\mathbf E$. With the help of Zorn's lemma, this amounts to the fact that every non-expansive map defined on  a subset $A$ of $E'$ extends to every $x\in E'\setminus A$ to a non-expansive map from the subspace of $\mathbf E'$  induced on $A\cup\{x\}$ to $\mathbf E$. \\
 Let us say that a \emph{closed ball} in $\mathbf E$  is any subset  of $E$ of the form $B(x, r):= \{y\in E: d(x,y)\leq r\}$. We say that the space $\mathbf E$ is \emph{hyperconvex}\;  if the
intersection of every family
of closed balls $\left(B (x_i,r_i)\right)_{i\in I}$ is non-empty whenever
$d(x_i,y_i) \leq r_i \oplus \overline {r_j}$ for all $i, j \in I$.

\subsection{How these notions relate?}

\begin{claim} For a  generalized metric space,  hyperconvexity implies  extension  property; extension property is  equivalent to injectivity, and  injectivity imply  that the space is an absolute retract.  \end{claim}

\noindent Hyperconvexity is equivalent to the conjunction of the following conditions:\\
1) \emph{Convexity}: for all $x,y \in E$ and $p,q \in   \mathcal H$ such
that $d(x,y) \leq p \oplus q$ there is $z \in E$ such that $d(x,z)\leq p$
and $d(z,y)\leq q$.\\
2)\emph{$2$-Helly property}, also called the
\emph{$2$-ball intersection property}: The intersection of every set
(or, equivalently, every family) of closed balls is non-empty provided
that their pairwise intersections are all non-empty.

\subsection{Heyting algebra alias dual integrale involutive quantale}

\begin{claim}\label{claim:distributivity}
The four notions above are equivalent  provided  that the set $\mathcal H$ of values of the distances is a complete lattice and satisfies   the following distributivity condition:

\begin{equation}\label{distribinfty}
\bigwedge  _{\alpha \in A, \beta \in B} u_\alpha \oplus  v_\beta =
\bigwedge _{\alpha \in A} u_\alpha  \oplus \bigwedge _{\beta \in B} v_\beta
\end{equation}
for all $u_\alpha \in \mathcal H$ $(\alpha \in A)$ and
$v_{\beta} \in \mathcal H$ $(\beta \in B)$.

\end{claim}

In this case,  we say  that $\mathcal H$ is an  \emph{involutive Heyting algebra}. 
The proof of Claim \ref{claim:distributivity} relies on Proposition \ref{prop-embedding} below.

\subsection{Metrisation of the set of values}

On an involutive  Heyting algebra $\mathcal H$, we may define a $\mathcal H$-distance. This is the most important fact about generalized metric spaces. It relies on the classical notion of residuation. Let $v\in \mathcal H$. Given $\beta \in \mathcal H$, the sets
 $\{r \in \mathcal H: v \leq r \oplus \beta\}$ and
$\{r \in \mathcal H: v \leq  \beta \oplus r \}$ have least elements, that we
denote
respectively  by $\lceil v\oplus -\beta\rceil$ and $\lceil -\beta\oplus  v  \rceil$ and call the \emph{right} and \emph{left  quotient} of $v$ by $\beta$ (note that  $\overline {\lceil -\beta\oplus v \rceil} =\lceil \bar v\oplus {-{\bar\beta}} \rceil$). It follows that for all
$p, q \in \mathcal H$, the set:

$$D(p,q):=\{r \in \mathcal H : p\leq q \oplus \bar r\;\;\text{and}\; \; q\leq
p\oplus r\} $$
has a least
element, namely $\lceil \bar p\oplus -\bar q \rceil \vee
\lceil -p\oplus q \rceil$, that we denote by $d_\mathcal H(p,q)$.

\begin{lemma} \cite{jawhari-al}
If $\mathcal H$ is a   Heyting algebra then for every every metric space $(E,d)$ over $\mathcal H$, and for all $x,y \in E$,
the following equality holds:
$$d(x,y) =  \bigvee_{z\in E} d_{\mathcal H}\left(d(z,x),d(z,y)\right).$$ 
\end{lemma}
Let $\left( (E_{i}, d_{i})\right) _{i\in I}$ be a family of $\mathcal H$-metric spaces. The \emph{direct product}  $\underset{i\in I}{%
\prod }\left( E_{i}, d_{i}\right) $, is the metric space $(E,d) $ where $E$ is the cartesian product  $%
\underset{i\in I}{\prod }E_{i}$ and $d$ is the  ''sup'' (or
$\ell ^{\infty }$) distance  defined
by $d\left(
\left( x_{i}\right) _{i\in I},\left( y_{i}\right) _{i\in I}\right) =%
\underset{i\in I}{\bigvee }d_{i}(x_{i},$ $y_{i})$.
We recall the following result. 

\begin{proposition} \label{prop-embedding}(Proposition II.2-7 of \cite{jawhari-al})The map $(p,q) \mapsto d_{\mathcal H}(p,q)$ over a Heyting algebra $\mathcal H$ is a distance
on $\mathcal H$, in fact $(\mathcal H,d_{\mathcal H})$ is an hyperconvex metric
space and every metric space over $\mathcal H$ embeds isometrically
into a power of  $(\mathcal H,d_\mathcal H)$. 
\end{proposition}

Hyperconvex spaces enjoy the extension property, hence they are injective. Since  hyperconvexity is preserved under the formation of products, Proposition \ref{prop-embedding} ensures that every  metric space embeds isometrically into an injective object. This fact is shortly expressed by saying  that the category of metric spaces over $\mathcal H$ has \emph{{enough injectives}}. From that follows 
 an  important structural property of the category of metric spaces over a Heyting algebra.

\begin{proposition} (Theorem 1, section II-2.9 of \cite{jawhari-al}) \label{equivalenceAR} In the category of metric spaces over  a Heyting algebra $\mathcal H$, injective, absolute retracts,  hyperconvex spaces, spaces with the extension property and retracts of powers of $({\mathcal H},d_{\mathcal H})$ coincide.\end{proposition}

\subsection{Injective envelope}In the category of  metric spaces (over the non-negative reals), every metric space has an injective envelope (also called an injective hull), a  major fact due to Isbell \cite{isbell}. One can view an injective envelope of a metric space $\mathbf E$ as an hyperconvex isometric extension $\mathbf F$ of $\mathbf E$, which is minimal with respect to inclusion (that is, there is no proper hyperconvex subspace of $\mathbf F$ containing isometrically $\mathbf E$).  One  can note that those minimal  extensions are isometric via  the identity on  $\mathbf E$. These facts extend to generalized metric spaces. 
\begin{proposition}\label{prop:injectiveenvelope} (Theorem 2, section II-3.1 of \cite{jawhari-al}). 
Every metric space over a Heyting algebra has an injective envelope. 
\end{proposition}
%
\subsection{Fixed point property}

An element $v\in \mathcal H$ is \emph{self-dual} if $\overline v=v$, it is   \emph{\emph{accessible}}\;  if there is some $r\in \mathcal H$ with $v\not \leq r$ and $v \leq r\oplus \overline r$ and \emph{\emph{inaccessible}}\;  otherwise. Clearly, $0$ is inaccessible;   every inaccessible element $v$ is self-dual (otherwise, $\overline v$ is incomparable to $v$ and we may choose $r:= \overline v$). We say that a space $\mathbf E$ is \emph{bounded}\;  if  $0$ is the only inaccessible element below the diameter $\delta(\mathbf E)$ of $\mathbf E$ (the \emph{diameter} of $\mathbf E$ is $\delta(\mathbf E):= \bigvee \{d(x,y): x,y\in E\}$).

%
\begin{theorem} \label{thm:cor3}
If a generalized   metric space  over a Heyting algebra is bounded and hyperconvex  then every commuting family  of non expansive self maps has a common fixed point. 
\end{theorem}

This result was obtained by J.B. Baillon \cite{baillon} for ordinary metric spaces. His proof applies without much changes.  It follows from a much more general result that we present in the next section. 
\subsection{Compact normal structure}
A generalization to   spaces with a compact and normal structure of the fixed point theorem of Sine and Soardi was obtained by Kirk \cite{kirk}, then M.A.Khamsi \cite{khamsi} extended to these spaces the result of Baillon.  
We extend first the notion of Penot \cite{penot} of compact normal structure to our spaces. 

A generalized metric space $\mathbf E$ has a \emph{compact structure }\;  if  the intersection of every family of closed balls  is nonempty provided that the  intersections of finite subfamilies are nonempty.

The \emph{diameter}\;  $\delta_{\mathbf E}(A)$ of a subset $A$ of a metric space $\mathbf E$ is 
$\bigvee \{d(x,y): x,y\in A\}$. The \emph{ radius }\;  $r_{\mathbf E}(A)$ of a subset $A$ is $\bigwedge\{ r \in \mathcal H: A\subseteq B(x,r)\;  \text{for some}\;  x\in A\}$.

 A subset $A$ of a metric space  $\mathbf E$ is \emph{equally centered}\;  if  $r_{\mathbf E} (A) =\delta_{\mathbf E}(A)$. For an example, if $A$ is the empty set and $E$ is nonempty   then $A$ is not equally centered. If $A$  a singleton, say $a$, then $A$ is equally centered.
 
 The space $\mathbf E$ has a  \emph{\emph{normal  structure}}\;  if no intersection of closed balls $A$ distinct from a singleton is equally centered. Equivalently, if  $|A|\not =­1$  then   $r_{\mathbf E} (A) \not =\delta_{\mathbf E}(A)$.

\begin{example} 

If $\mathbf E$ is a hyperconvex metric space, it has a compact structure; if it is bounded it has a normal structure (in fact, the radius is half the diameter). 

\end{example}
\begin{lemma}
Let $A$ be an intersection of balls  of $\mathbf E$. If $\delta_{\mathbf E}(A)$ inacessible  then  $A$ is equally centered; the  converse holds if $\mathbf E$ is hyperconvex. 
\end{lemma}

Khamsi and the third  author of this paper obtained the following result \cite{khamsi-pouzet}: 
\begin{theorem}
If a generalized   metric space   has a compact and normal structure  then every commuting family  of non expansive self maps has a common fixed point. 
\end{theorem}

The key of their proof is the notion of one-local retract, already used in \cite{khamsi}. 
We say that $\mathbf A:=(A, d_{\restriction A})$  is a \emph{one-local retract} of $\mathbf E$ if it is a retract of $(A\cup \{x\},  d_{\restriction A\cup \{x\}})$ (via the identity map) for every $x\in E$.

Adapting the intersection process discovered by  Baillon \cite{baillon}, they proved:

\begin{theorem}\label{thm:best} If a  generalized metric space $\mathbf E$ has a compact normal structure,  then, the intersection of every down-directed family $\mathcal F$ of one-local retracts of $\mathbf E$ is a nonempty one-local retract of $\mathbf E$. 
 \end{theorem}
 
 We stop here the description of categorical properties of generalized metric spaces. The interested reader will find more, notably on Hole-preserving maps and one-local retracts in \cite{jawhari-al}  and in \cite{kabil-pouzet2}. 
 
 In  \cite{pouzet-rosenberg} there is a  study of more general metric spaces for which the neutral element, $0$  of the set $\mathcal H$ of values is not necessarily the least element  and condition $(i)$ for the distance is replaced $d(x,y)  = 0$ iff $x=y$. Despite of the scope and applicability we volontarily omitted it.

\section{From generalized metric spaces to graphs, ordered sets and automata} 

We illustrate the notion of generalized metric spaces with graphs, ordered sets and a special kind of transition systems. We conclude this section with an application of the notion of injective envelope to the freeness of a monoid of final segments. 
\subsection{The case of ordinary metric spaces} 
Let $\mathbb R^{+}$ be the set of non negative reals with the addition and natural order, the involution being the identity. Let $\mathcal H$ be $\mathbb R^{+} \cup \{+\infty\}$. Extend to $\mathcal H$ the addition and  order in a natural way. Then, metric spaces over $\mathcal H$ are direct  sums of ordinary metric spaces (the distance between elements in different components being $+\infty$). The set $\mathcal H$ is a Heyting algebra  and  the distance $d_{\mathcal H}$ once restricted to $\mathbb R^{+}$  is the absolute value. The inaccessible  elements are $0$ and $+\infty$ hence, if one  deals with ordinary metric spaces,  unbounded  spaces in the above sense are those which are unbounded in the ordinary sense.  If one deals with ordinary metric spaces, infinite products  can yields spaces  for which $+\infty$ is attained. Thus, one has to  replace powers of $\mathbb R^{+}$ by  $\ell^{\infty}$-spaces (if $I$ is any set, $\ell^{\infty}_{\mathbb R} (I)$ is the set of bounded families  $(x_i)_{i\in I}$ of reals numbers, endowed with the sup-distance). With that, the notions of  absolute retract, injective, hyperconvex and retract of some $\ell^{\infty}_{\mathbb R}(I)$ space coincide. This is the well known  result of Aronszjan-Panitchpakdi \cite{aronszajn-panitchpakdi}. The existence of an injective envelope was proved by Isbell \cite{isbell}. The injective envelope of a $2$ element ordinary metric space is a bounded closed interval of the real line; injective envelopes of ordinary metric spaces consisting of few many elements have been described \cite{dress}; for applications see \cite{chepoi}. The existence of a fixed point for a non-expansive map on a bounded hyperconvex space is the famous result of Sine and Soardi \cite{sine, soardi}. Theorem \ref{thm:cor3} applied to a bounded  hyperconvex metric space is  Baillon's fixed point theorem \cite{baillon}.  Applied to 
a metric space with  a compact  normal structure,  this is the result obtained by  Khamsi \cite{khamsi}(1996). 

 The results presented about generalized metric spaces over a Heyting algebra   apply to ultrametric spaces  over $\mathbb R^{+} \cup \{+\infty\} $.  Indeed, with a the join operation, the distributivity condition holds, hence $\mathbb R^{+} \cup \{+\infty\} $ is a Heyting algebra. A similar characterization to ours  was obtained in \cite{bayod-martinez};  a description of the injective envelope  is also given.  The paper \cite{pouzet-rosenberg} contains a study of ultrametric spaces  over a complete lattice satisfying this  distributivity condition, called an \emph{op-frame}. Metric spaces over op-frame are  studied in \cite{ackerman}. Ultrametric spaces  over a lattice and their  connexion with collections of equivalence relations have been recently studied in \cite{braunfeld}.  More general ultrametric spaces have been studied in \cite{priess-crampe-ribenboim1, priess-crampe-ribenboim2, priess-crampe-ribenboim3}. Due to their interest, we  devote  most of the last  section of this paper to their study.

\subsection{Directed  graphs, transition systems and ordered sets}
Let us equip  directed graphs with the zizag distance. The set $\mathbf F(\Lambda^*)$ of values of the distance is an involutive Heyting algebra. We may apply the results of the theory.

\begin{lemma} A map from a reflexive directed graph $G$ into an other is a graph-homomorphism iff it is non-expansive w.r.t. the zigzag distance. 
\end{lemma}

\begin{lemma}\label{lem:connexity}
The distance $d$ of  a metric space $(E,d)$ over $\mathbf {F}({\Lambda^*})$ is the zigzag distance of some  reflexive directed graph $G:= (E, \mathcal E)$ iff  it satisfies the following property for all  $x,y,z \in E$, $u,v\in  \mathbf F(\Lambda^*)$:
 $u.v \in d(x,y)$ implies $u\in d(x,z)$ and $v\in d(z,y)$ for some $z\in E$. When this condition holds, $(x,y)\in \mathcal E$ iff $+\in d(x,y)$.
\end{lemma}

Due to Lemma \ref {lem:connexity} above, the various metric spaces mentionned above (injective, absolute retracts, etc.) are graphs equipped with the zigzag distance; in particular, the distance $d_{\mathbf F({\Lambda^*})}$  defined on $\mathbf F({\Lambda^*})$ is the zigzag distance of some graph. This later  fact leads to a fairly precise description of absolute retracts in the category of reflexive directed graphs (see \cite{kabil-pouzet}).

\subsection{Transition systems} Instead of a two-letters alphabet, we may consider a finite one, say $A$. The analog of directed graphs are transition systems.   A   \emph{transition sytems}  is a pair $\mathcal T:= (Q, \mathcal T)$, where the elements $q\in Q$ are called  \emph{states}  and the elements  of $T$, the \emph{transitions},  are triples $(p, a, q)\in Q\times A\times T$.

If $x$ and $y$ are two states, we may define the distance from $x$ to $y$ as the set $d_{\mathcal T} (x,y)$ of words coding the paths from $x$ to $y$. In automata theory this is simply   the \emph{ language accepted by the automaton}  made of $\mathcal T$, in-state $x$ and out-state $y$. 

To mimic the case of directed graphs, we could equip the alphabet $A$ of an involution $-$ and impose our transition systems to be\;  \emph{involutive}\; that is for every letter $a$, 
$(p, a, q)\in T$ iff $(p, \overline a, q) \in T$. This is a cosmetic change in the usual theory of languages. We could  impose the system to be\;  \emph{reflexive}\; that $(p,\alpha, p)\in T$ for every state $p$, letter $a$. This is a strong requirement, about the same than imposing our graphs to be reflexive. We refer to \cite{pouzet-rosenberg} for more. 

\subsection{Ordered sets}  In this case, zizags reduce to \emph{fences}.
There are  two fences of length $n$: the \emph{up-fence}; and the  \emph{down-fence}. The first one starts with $x_0<x_1>...$, the second with $x_0>x_1<..$. So one can express the distance as the  pair $(n,m)$ of integers such that $n$ is the  shortest length of an up-fence from $x$ to $y$ and $m$ the  shortest length of a down-fence from $x$ to $y$. For more, see Nevermann-Rival, \cite{nevermann-rival} 1985 and Jawhari-al  \cite{jawhari-al}1986.

\subsection{The case of oriented graphs}\label{subsection:orientedgraphs}
Oriented graphs and directed graphs behave  differently. Oriented  graphs cannot be modeled over a Heyting algebra (Theorem I V-3.1 of  \cite{jawhari-al} is erroneous), but the absolute retracts in this category can be (this was proved by Bandelt, Sa\"{\i}dane and the third  author of this paper and included in Sa\"{\i}dane's thesis \cite{saidane}). The appropriate Heyting algebra is $\mathbf N(\Lambda^*)$, the  \emph{MacNeille completion} of $\Lambda^{\ast}$.

The MacNeille completion of $\Lambda^{\ast}$ is in some sense the least complete lattice extending $\Lambda^{\ast}$. The definition goes as follows. If $X$ is a subset of $\Lambda^{\ast}$ ordered by the subword ordering then
 $$X^{\Delta}:= \bigcap_{x \in X} \uparrow x$$
is the {\it upper cone} generated by $X$, and
$$X^{\nabla}:= \bigcap_{x \in X} \downarrow x$$
is the {\it lower cone} generated by $X$.

The pair $(\Delta, \nabla)$ of mappings on the complete
lattice of subsets of  $\Lambda^{\ast}$ constitutes a Galois connection. Thus,  a set $Y$ is an upper cone  if
and only if $Y = Y^{\nabla \Delta}$, while a set $W$ is
an lower cone if and only if $W = W^{\Delta \nabla}.$ This Galois connection
$(\Delta, \nabla)$ yields the {\it MacNeille completion} of
$\Lambda^{\ast}.$ This completion is realized  as the complete
lattice $\{W^{\nabla}:  W\subseteq \Lambda^{\ast}\}$
ordered by inclusion or alternatively $\{Y ^{\Delta} : Y\subseteq \Lambda^{\ast}\}$ ordered by reverse inclusion. We choose as completion the set $\{Y ^{\Delta} : Y\subseteq \Lambda^{\ast}\}$ ordered by reverse inclusion  that we denote by $\mathbf {N}(\Lambda^{\ast})$. This complete lattice is studied in detail  in  Bandelt and Pouzet, 2018 \cite{bandelt-pouzet}.

We recall the following characterization of members of the MacNeille completion of $\Lambda^*$. 
\begin{proposition}\label{cancellation} \cite{bandelt-pouzet} Corollary 4.5.
A member  $Z$ of $\mathbf{F}(\Lambda^{*})$ belongs to $\mathbf {N}(\Lambda^{*})$ if and only if it satisfies the following cancellation rule:
if $u + v \in Z$ and $u - v \in Z$ then $u v \in Z$. 
\end{proposition}

 The concatenation, order and involution defined on $\mathbf {F}(\Lambda^{\ast})$ induce an involutive  Heyting algebra  on  $\mathbf {N}(\Lambda^{\ast})$ (see Proposition 2.2 of \cite{bandelt-pouzet}). Being an involutive Heyting algebra, $\mathbf N({\Lambda^*})$ supports a  distance $d_{\mathbf N({\Lambda^*})}$ and this distance is the zigzag distance of a graph $G_{\mathbf {N}({\Lambda^*})}$. But it is not true that every oriented graph embeds isometrically into a power of that graph. For example, an oriented cycle cannot be embedded.  The following result characterizes  graphs which can be  isometrically embedded, via the zigzag distance, into  products of reflexive and oriented zigzags. It is stated in part in Subsection IV-4 of \cite{jawhari-al}.

\begin{theorem}\label{theo:isometric}
	For a  directed graph   $G$ equipped with the zigzag distance,  the following properties are equivalent:
\begin{enumerate} [(i)]
	\item $G$ is isometrically embeddable into a product of  reflexive and oriented zigzags;
	\item $G$ is isometrically embeddable into a power of $G_{\mathbf  N({\Lambda^*})}$;
	\item The values of the zigzag distance between  vertices of $V$ belong to $\mathbf N({\Lambda^*})$.
\end{enumerate}
\end{theorem}

\begin{theorem}\label{thm:ARcompletion}
An oriented graph $G$ is an absolute retract in the category of oriented graphs if and only if it is a retract of a product of  oriented zigzags.\end{theorem}

 We just give a sketch. For details,  see Chapter V of \cite{saidane} and the  forthcoming paper of Bandelt, Pouzet, Sa\"{\i}dane \cite{bandelt-pouzet-saidane}. The proof has three  steps. Let $G$ be an absolute retract. First, one proves that $G$ has no $3$-element cycle. Second, one proves that the zigag distance between two vertices of $G$ satisfies the cancellation rule. From Proposition \ref{cancellation},  it belongs to $\mathbf{N}(\Lambda^*)$; from Theorem \ref{theo:isometric}, $G$ isometrically embeds into a product of oriented zigzags. Since $G$ is an absolute retract, it is a retract of that product.

As illustrated by the results of Tarski and Sine and Soardi, absolute retracts are appropriate candidates for the fixed point property. Reflexive graphs with the fixed point property must be antisymmetric, i.e., oriented. Having described absolute retracts among oriented graphs,  we derive from Theorem \ref{thm:cor3} that the  bounded ones have the fixed point property. 

We start with a characterization of accessible elements of $\mathbf N({\Lambda^*})$. The proof is omitted. 
\begin{lemma} \label{lem-accessible}Every element $v$ of $\mathbf N({\Lambda^*})\setminus \{\Lambda^{\ast}, \emptyset\}$ is accessible.
\end{lemma}
\begin{theorem}\label{thm:cor4}
If a graph $G$, finite or not, is a retract of a product of reflexive and directed zigags of bounded length then every commuting set of endomorphisms has a common fixed point.
\end{theorem}
\begin{proof}
We may suppose that $G$ has more than one vertex. The diameter of $G$ equipped with the zigzag distance belongs to $\mathbf N_{\Lambda^*}\setminus \{\Lambda^{\ast}, \emptyset\}$. According to  Lemma \ref{lem-accessible},  it is accessible, hence as a metric space,  $G$ is bounded. Being a retract of a product of hyperconvex metric spaces it is hyperconvex. Theorem \ref{thm:cor3} applies.
\end{proof}

If we  consider  zigzags of length one we get Tarski's fixed point theorem \cite{tarski}.

\subsection{An illustration: the freeness of $\mathbf F(A^*)$ and $\mathbf N(A^{*})$}

Instead of a two letter alphabet, we consider un arbitrary alphabet $A$, not necessarily finite. We suppose  that the alphabet $A$ is
ordered. We
order  $A^{\ast }$ with the \emph{Higman ordering }: if $\alpha $ and
$\beta $ are two elements in $A^{\ast }$ such $\alpha: =a_{0}\cdots 
a_{n-1}$ and $\beta: =b_{0}\cdots  b_{m-1}$ then $\alpha \leq \beta$
if there is an injective and increasing map $h$ from $\left\{
0,...,n-1\right\} $ to $\left\{ 0,...,m-1\right\}$ such that for each $i$, $0\leq i\leq n-1$, we have $a_{i}\leq b_{h\left( i\right) }$. Then
$A^{\ast }$ is an ordered monoid with respect to the concatenation
of words.  
A \emph{\emph{final segment}} of $A^{\ast}$ is any subset $F\subseteq A^{\ast}$ such that $\alpha \leq
\beta,\alpha \in F$ implies $\beta \in F$.  Initial segments are defined dually.  Let $X$ be a subset of 
$A^{\ast}$; then $$\uparrow X:= \{ \beta  \in A^{\ast}: \alpha \leq \beta\;  \text{for some}\;  \alpha\in X\}$$  is the {\it upper set} generated by $X$ and 
$$\downarrow X := \{\alpha  \in A^{\ast}:  \alpha \leq \beta \;  \text{for some}\;  \beta\in X\}$$   is the {\it lower set} generated by $X$.

 Let $\mathbf F(A^{\ast})$ be the set of final segments of $A^{\ast}$. The  concatenation of  words extends to $\powerset (A^{\ast})$; this operation defined by $XY:= \{\alpha\beta: \alpha\in X, \beta\in Y\}$ induces an operation on $\mathbf F(A^{\ast})$  for which the set $A^{\ast}$ is neutral. Hence $\mathbf F(A^{\ast})$ is a monoid.  Since it contains the empty set $\emptyset$ and $\emptyset$ has several decompositions (e.g. $\emptyset=\emptyset A^{\ast}= A^{\ast}\emptyset $), this monoid is not free.  Let  $\mathbf F^{\circ}(A^{\ast}):= \mathbf F(A^{\ast})\setminus \{\emptyset \}$ be the set of non-empty final segments of $A^{\ast}$.  This is submonoid of $\mathbf F(A^{\ast})$.%

\begin{theorem}\label{prop1}
 $\mathbf F^{\circ}(A^{\ast})$ is a free monoid.
 \end{theorem}

 The following  illustration of Theorem \ref{prop1} was proposed to us by J.Sakarovitch. An \emph{\emph{antichain} } of  $A^{\ast}$ is any subset $X$ of $A^{\ast}$ such that any two distinct elements $\alpha$ and $\beta$ of $X$ are incomparable w.r.t. the Higman ordering. The set  $Ant (A^{\ast})$ of antichains of $A^{\ast}$ and the set $Ant_{<\omega}(A^{\ast})$ of finite antichains of $A^{\ast}$ are  submonoids of $\powerset( A^\ast)$; the sets $Ant^{\circ} (A^{\ast}):= Ant (A^{\ast})\setminus \{\emptyset\}$ and 
 $Ant_{<\omega}^{\circ}(A^{\ast}):= Ant_{<\omega}(A^{\ast})\setminus \{\emptyset\}$ of non-empty antichains are also submonoids. 
 From Theorem \ref{prop1}, we deduce:
 \begin{theorem} \label{antichainfree}
 The monoids  $Ant^{\circ} (A^{\ast})$ and $Ant^{\circ}_{<\omega}(A^{\ast})$  are free. 
 \end{theorem}
 
\subsubsection{Well-quasi-ordered alphabets.}
 Note that if $A$ is \emph{\emph{well-quasi-ordered} }(w.q.o)(that is to say that every 
  final segment of $A$ is finitely generated)  then the  monoids $Ant (A^{\ast})$ and $Ant_{<\omega}(A^{\ast})$ are equal and isomorphic to the monoid $F(A^{\ast})$, thus Theorem  \ref{antichainfree} reduces to  Theorem \ref{prop1}. Indeed, if $A$ is w.q.o. then, according to a famous result of Higman \cite {higman} 1952, $A^{\ast}$ is w.q.o. too,  that is  every final segment $F$ of $A^{*}$ is generated by $Min (F)$ the set of minimal elements of $F$. Since $Min(F)$ is an antichain and in this case a finite one, our claim follows.

\subsubsection{The MacNeille completion of $A^{\ast}$.}
 
  Let $\mathbf N(A^{\ast})$ be the MacNeille completion of the poset $A^{\ast}$, that we may view as the collection of intersections of principal final segments of $A^{\ast}$. The MacNeille completion of $\mathbf N(A^{\ast})$ is a submonoid of $\mathbf F(A^{\ast})$. From Theorem \ref{prop1}, we derive: 
\begin{theorem}\label{thm:completionfree} Let $A$ be an ordered alphabet. The monoid $\mathbf N^{\circ}(A^{\ast}):= \mathbf N(A^{\ast})\setminus \{\emptyset\}$ is free.
\end{theorem}
We recall that a member $F$ of $\mathbf F(A^{\ast})$ is  \emph{irreducible} if it is distinct from $A^{\ast}$ and  is not the concatenation of two members  of $\mathbf F(A^{\ast})$ distinct of $F$ (note that with this definition, the empty set is irreducible). For an example, if  $F= \uparrow \{u, v\}$ with  $u$  incomparable to $v$, then $F$ is  irreductible iff  $u$ and $v$ do not have a common prefix nor a common suffix.
 The fact that $\mathbf F^{\circ}(A^{\ast})$ is free amounts to the fact that each  member decomposes in a unique way as a concatenation of  finitely many irreducible elements. 
  
\subsubsection{An interpretation}

We interpret the freeness of these monoids  by means of injective envelopes of $2$-element metric spaces. 


We suppose that $A$ equipped with an involution (this is not a restriction: we may choose the identity on $A$ as our involution). We consider  a notion of metric spaces with values in $\mathbf F(A^{\ast})$. The category of these spaces over $\mathbf F(A^{\ast})$, with the non-expansive maps as morphisms,  has enough injectives (meaning that every metric space extends isometrically to an injective one).  For every  final segment $F$ of $A^{\ast}$, 
 the $2$-element space metric space $E:= (\{
x,y\}, d)$ such that $d(x,y)=F$,  has an \emph{{injective envelope}}\;  $\mathcal S_F$ (that is a minimal extension to an injective metric space).

Since $\mathcal S_F$ is injective, corresponds  to it a  transition system  $\mathcal {M}_{F}$ on the alphabet $A$, with transitions $(p, a, q)$ if $a \in d(p,q)$. The automaton  $\mathcal{A}_{F}:=  ({\mathcal M}_{F},\left\{ x\right\}, \left\{y\right\} )$ with $x$ as initial state and $y$ as final state accepts $F$.  A transition system yields a directed graph  whose arcs are the ordered pairs $(x,y)$ linked by a transition. The transition system $\mathcal {M}_{F}$ being reflexive and involutive,  the corresponding  graph $\mathcal {G}_{F}$ is undirected and has  a loop  at every vertex.  For an example, if $F= A^{\ast}$, $\mathcal S_F$ is the one-element metric space and $\mathcal {G}_{F}$ reduces to a loop. If $F= \emptyset$, $\mathcal S_F$ is the two-elements  metric space $E:= (\{x, y\}, d)$ with $d(x,y)= \emptyset$ and $\mathcal {G}_{F}$ has no edge.

The gluing of two injectives by a common vertex  yields an injective; we will  say that an injective which is not the gluing of two proper injectives is \emph{irreducible}.

With the notion of cut vertex and block borrowed from graph theory, Kabil, Pouzet and  Rosenberg \cite{KPR} proved: 

\begin{theorem}\label{thm:blockdecomposition}
Let $F$ be a  final segment of $A^{\ast}$, distinct from $A^{\ast}$. Then $F$ is irreducible if and only if  $\mathcal S_F$ is irreducible if and only if $\mathcal {G}_F$ has no cut vertex. If $F$ is not irreducible, 
the blocks of $\mathcal {G}_F$ are  the vertices of a finite path $C_0, \dots, C_{n-1}$ with $n\geq 2$,  whose  end vertices  $C_0$ and $C_{n-1}$ contain respectively  the initial state $x$ and the final state $y$ of the automaton $\mathcal{A}_{F}$ accepting $F$. Furthermore, $F=F_0 \dots  F_i \dots  F_{n-1}$, the automaton $\mathcal{A}_{F_i}$ accepting $F_i$  being isomorphic to  $({\mathcal M}_{F}\restriction C_{i},\left\{ x_i\right\}, \left\{
y_i\right\} )$, where $x_i:=x$ if $i=0$, $y_{i}= y$ if $i=n-1$ and  $\{x_i\}=C_{i-1} \cap C_{i}$, $\{y_i\}= C_{i} \cap C_{i+1}$, otherwise. 
\end{theorem} 

From this result, the freeness of $\mathbf F^{\circ}(A^{\ast})$ follows. 

\section{Preservation of binary relations by operations}
We recall the duality between relations and operations. We consider then the special case of binary relational structures, particularly equivalence relations, leading to generalized metric   and ultrametric spaces.

\subsection{Duality between relations and operations}
Let $E$ be a set.   For $n\in \N^*:= \N\setminus
\{0\}$, a map $f: E^n
\rightarrow E$ is an
$n$-{\it ary operation} on $E$, whereas a subset
$\rho \subseteq E^{n}$ is an $n$-{ary relation} on $E$.
Denote by
$\mathcal O^{(n)}$ (resp.$\mathcal R^{n}$) the set of
$n$-ary operations (resp. relations)  on $E$ and set
$\mathcal O:=\bigcup \{\mathcal O^{(n)}: n\in N^* \}$ (resp
$\mathcal R:= \bigcup \{\mathcal R^{(n)}:
n\in N^* \}$). For  $n, i\in N^*$  with $i\leq n$, define the   $i^{th}$
$n$-{\it ary projection}
$e^{n}_{i}$ by setting $e^{n}_{i}(x_{1},\dots,x_{n}):= x_{i}$ for all  $x_{1},\dots,x_{n}\in E$ and set
$\mathcal P:= \{e^{n}_{i}: i,n \in \N^*\}$. An operation $f\in  \mathcal O$ is {\it constant} if it takes a single value, it
 is\ {\it idempotent} provided
$f(x,\dots,x)= x$ for all $x\in E$. We denote by  $\mathcal C$ (resp. $\mathcal I$) the set of  constant, (resp. idempotent)  operations on $E$.

Let $m,n\in \N^*$,
$f \in \mathcal O^{(m)}$ and $\rho \in \mathcal R^{(n)}$.    Then $f$ {\it preserves}
$\rho$
 if:
\begin{equation}
\small {(x_{1,1}, \dots, x_{1,n})\in
\rho, \dots,
(x_{m,1}, \dots, x_{m,n})\in \rho \Longrightarrow (f(x_{1,1}, \dots, x_{m,1}),\dots,f(x_{1,n}, \dots,
x_{m,n}))\in \rho}
\end{equation}
for every $m\times n$ matrix
$X:= (x_{i,j})_{i=1,\ldots,m \atop {j=1, \ldots ,n}}$ of elements of $E$.

If $\rho$ is binary and $f$ is unary, then $f$ preserves $\rho$ means:

\begin{equation}
(x, y)\in \rho
 \Longrightarrow (f(x),f(y))\in \rho
\end{equation}

for all $x, y \in E$.


If $\mathcal F$ is a set of operations on $E$, let $\Inv(\mathcal F)$, resp. $\Inv_n(\mathcal F)$ be the set of relations, resp. $n$-ary relations,  preserved by all $f\in \mathcal F$. Dually, if $\mathcal R$ is a set of relations on $E$, let $\Pol(\mathcal R)$, resp. $\Pol_n(\mathcal R)$,  be the set of operations, resp. $n$-ary operations,  which preserve all $\rho\in \mathcal R$. The operators $\Inv$ and $\Pol$ define a Galois correspondence. The study of this correspondence is the theory of clones \cite{lau}.

Two basic problems have been considered:

\noindent 1) Describe the sets of the form $Inv(\mathcal F)$. 

\noindent 2)  Describe the sets of the form $Pol (\mathcal R)$. 

A solution is due to  Bodnar\u{c}uk, Kalu\u{z}nin, Kotov,
and Romov \cite{BKKR69a, BKKR69b}. 
%
%
In concrete cases, this description does not  help much. For example, given  $\mathcal R$, decide if $Inv(Pol(\mathcal R))=\mathcal R$?

\subsection{Towards generalized metric spaces} We restrict our attention to the  case of unary operations and binary relations. We recall that if $\rho$ and $\tau$ are two binary relations on the same set $E$, then their composition $\rho\circ\tau$ is the binary relation made of pairs $(x,y)$ such that $(x,z)\in \tau$ and $(z,y)\in \rho$ for some $z\in E$. It is customary to denote it $\tau\cdot \rho$.

The set $Inv_2(\mathcal F)$ of binary relations on $E$ preserved by all $f$ belonging to a set $\mathcal F$ of self maps has some very simple  properties that we state below (the proofs are left to the reader).  For the construction of many more properties by means of primitive positive formulas, see \cite{snow}.

\begin{lemma}\label{lem:monoid} Let $\mathcal F$ be a set of unary operations on a set $E$.
Then the set  $\mathcal R:=\Inv_2(\mathcal F)$ of binary relations on $E$ preserved by all $f\in \mathcal F$ satisfies the following properties:

\begin{enumerate}[{(a)}]
\item $\Delta_E\in \mathcal R$;
\item $\mathcal R$ is closed under arbitrary intersections; in particular $E\times E\in \mathcal R$;
\item $\mathcal R$ is closed under arbitrary unions;
\item If $\rho, \tau \in \mathcal R$ then $\rho \cdot \tau \in \mathcal R$;
\item If $\rho\in \mathcal R$ then $\rho^{-1}\in \mathcal R$.

\end{enumerate}
\end{lemma}

Let $\mathcal R$ be a set of binary relations on  a set $E$ satisfying  items $(a)$, $(b)$, $(d)$ and $(e)$ of the above lemma (we do not require $(c)$). To make things more transparent, denote by $0$ the set $\Delta_E$, set $\rho\oplus \tau:= \rho\cdot \tau$. Then $\mathcal R$ becomes  a monoid.  Set $\overline {\rho}:= \rho^{-1}$, this defines an involution on $\mathcal R$ which reverses the monoid operation. With this involution $\mathcal R$ is an \emph{involutive monoid}. With the inclusion order, that we denote  $\leq$,  this involutive   monoid is an \emph{involutive complete ordered monoid}.

With these definitions, we have immediately:

\begin{lemma} \label{lem:distance}Let $\mathcal R$ be an involutive  complete ordered monoid of the set of binary relations on $E$ and let
$d$ be the map from $E\times E$ into $\mathcal R$ defined by  $$d(x,y):= \bigcap \{\rho\in \mathcal R: (x,y)\in \rho\}.$$
Then, the following properties hold:

\begin{enumerate}[{(i)}]
\item $d(x,y)  \leq 0$ iff $x=y$;
\item $d(x,y) \leq d(x,z)\oplus d(z, y)$;
\item $\overline{d(y,x)}=d(x,y)$.
\end{enumerate}
\end{lemma}

In
\cite{pouzet-rosenberg}, a set $E$ equipped with a map $d$ from $E\times E$ into an involutive ordered monoid $\mathcal H$ (for which  $0$ is not necessarily the least element of the monoid) and which satisfies   properties $(i), (ii), (iii)$     stated in Lemma \ref {lem:distance} is called a \emph{$\mathcal H$-distance}, and the pair $(E,d)$ a \emph{$\mathcal H$-metric space}.  A study of  metric spaces and non-expansive mappings along the lines of the one developed in Section \ref{section:generalized metric spaces} is in \cite{pouzet-rosenberg}.  It is no more than the  study of systems of binary relations and homomorphisms.  Indeed to a metric space $\mathbf E:=(E, d)$ over an ordered monoid $\mathcal H$, we may associate the relational structure $\mathbf R_d:= (E, (d_v)_{v\in \mathcal H})$  where $d_v:= 
\{(x,y): d(x,y)\leq v\}$. 
If $f$ is a map from $\mathbf E:=(E,d)$ into $\mathbf E:=(E',d')$, then $f$ is non-expansive iff $f$ is a homorphism from $\mathbf R_d$ into $\mathbf R_d'$; that is for every $v\in \mathcal H$, $(x,y)\in d_v$ implies $(f(x),f(y))\in d'_v$.  Lemma \ref {lem:distance}  justify that  we write  $d(x,y) \leq \rho$ for the fact that  a pair $(x,y)$ belongs to a binary relation $\rho$ on the  set $E$.  Hence, one can   use notions borrowed from the theory of metric spaces in the study of binary relational structures. An illustration is given in \cite{khamsi-pouzet}. 

In the next section, we consider equivalence relations and generalized ultrametric spaces.  



\section{Preservation of equivalence relations}

The presentation of this section is borrowed from \cite{pouzet}.

 A  binary relation  $\rho$ on a set $A$ is an \emph{{equivalence relation}}\;  if it is \emph{ reflexive}, \emph{symmetric}\; and \emph{transitive}. It decomposes $A$ into blocks. Two elements in the same block are equivalent; whereas two elements into two different blocks are inequivalent. The fact that a binary operation $f$ preserves $\rho$ means that if $x$ and $y$ belong to  blocks $X$ and $Y$ respectively, then the block containing $f(x, y)$ does not depends upon the particular choice of $x$ and $y$. This allows to define an operation on the set of blocks that mimicks  $f$.

  A pair $\mathcal A_F$ made a set $A$ and a  collection $\mathcal F$ of  operations on $A$ is called an \emph{algebra}. Equivalence relations preserved by all members  of $\mathcal F$ are called \emph{congruences} and their set denoted by $\Cong(\mathcal A_{\mathcal F})$. The study of the relationship between the set of congruences of an algebra and the set of maps which preserve all congruences is one of the goals of universal algebra.

\subsection{Algebra and congruences}

If $\mathcal F$ is a set of maps on $A$, the  set  $\Cong(\mathcal A_{\mathcal F})$ is a subset of the set $\Eqv(A)$ of equivalence relations on $A$; this set  is closed under intersection and union of chains. Ordered by inclusion this is an algebraic lattice. It was show by Gr\"atzer and Schmidt \cite{gratzer-schmidt} that every algebraic lattice is isomorphic to the congruence lattice of some algebra.

One of the oldest unsolved problem in  universal algebra is "the finite lattice representation problem":

\begin{problem}

Is  every finite lattice  isomorphic to the congruence lattice of  a \emph{finite} algebra? 

\end{problem}

See Gr\"atzer \cite{gratzer} 2007 and P\`alfy \cite{palfy2} 2001 for an overview. Say that a lattice $L$ is \emph{representable} as a congruence lattice if it is isomorphic to the lattice of congruences of some algebra.  Say that it is \emph{strongly representable} if every  sublattice $L'$ of some $\Eqv(A)$ (with the same $0$ and $1$ elements) which is isomorphic to $L$ is the lattice of congruences of some algebra on $A$. As shown in \cite{quackenbush-wolk}, not every representable lattice is strongly representable.
The first step in the positive direction for the representation problem  is the fact that every finite lattice  embeds as  a sublattice of the lattice of equivalences on a finite set,  a famous and non trivial result of Pudlak and Tuma \cite{pudlak-tuma} solving an old conjecture of Birkhoff. 

For an integer $n$, let $M_n$ be the lattice made of a bottom and a top element and an $n$-element antichain. Let $M_3$ be the lattice made of a $3$-element antichain and a top and bottom. This lattice is representable (as the set of congruences of the group $\Z/2\cdot\Z\times \Z/2\cdot\Z$) but not strongly representable. We may find sublattices $L$ of $\Eqv (A)$ isomorphic to $M_3$ such that the only unary maps preserving $L$ are the identity and constants (see Section \ref{section rigidity}). Hence, the congruence lattice of the algebra on $A$ made of these unary maps is $\Eqv (A)$. The sublattices $L$ of $\Eqv(A)$ such that $\Cong(\mathcal A_L)=\Eqv(A)$ (where $\mathcal A_L:=(A,\Pol^{1}(L))$)  are said to be \emph{dense}.   The fact that,  as a lattice,  $M_3$ has a dense representation in every $\Eqv(A)$ with $A$ finite on at least five elements, amounting to a Z\'adori's result \cite{Zad83} given is Section \ref{section rigidity},  appears  in \cite{demeo} as Proposition 3.3.1 on page 20. It is not known if  $M_n$ is representable for each  integer $n$ (it is easy to see that $M_n$ is representable if $n=q+1$ where $q$ is a power of a prime. The case $n=7$ was solved by W.Feit, 1983). 

\subsection{Orthogonal systems of equivalence relations}
Two equivalence relations $\rho$ and $\tau$ on the same set $E$ are \emph{orthogonal} if their intersection $\rho \cap \tau$ is the equality relation $\Delta_E$ and their join in the lattice $\Eqv(E)$ of equivalence relations on $E$ is the full relation $E\times E$ (Note that in \cite {MPRT} p.397 this terminology  was used for the fact that $\rho \cap \tau= \Delta_E$). 
 \begin{problems}
 \begin{enumerate} 
 \item  Given an integer $n$, find the largest number of  pairwise orthogonal equivalence relations on a  set of size $n$.
\item Find  the largest number of pairwise orthogonal equivalence relations on a  set of size $3m$ whose blocks have three elements?
 \item More generally, given two integers $k$ and  $m$,  find the largest integer $o(k, m)$ such that there are $o(k, m)$ pairwise orthogonal relations on a set of size $km$ whose blocks have $k$ elements?

\end{enumerate}
 \end{problems}
 
 The first question amounts to find the largest $m$ such that the lattice $M_m$ embeds as a sublattice of the lattice $\Eqv (E)$ of equivalences on a set of cardinality $n$. 
 The second question was asked to us by Rosenberg in October 2013. It is worth noticing that for $k= 2$, it was conjectured by Kotzig \cite{kotzig} in 1963 that $o(2, m)$ is $2m-1$. This conjecture, still unsolved, is known under the name of  \emph{Perfect one-factorisation (PIF)} of the complete graph $K_{2m}$. A \emph{one-factor}  of a graph is a set of pairwise disjoint edges whose union covers all vertices; a \emph{factorisation} is a covering all the edges by pairwise disjoint one-factors; a one-factorization is \emph{perfect} if the union of any two one-factors forms a Hamiltonian cycle.  Kotzig's conjecture is known to hold if $m$ or $2m-1$ is prime  and also if $2m$ is among some particular set of values; eg a PIF of $K_{56}$ was obtained  only this year \cite{pike}.

%

\subsection{Generalized ultrametric spaces}
Generalized ultrametric spaces provide natural sets of equivalence relations. We restrict ourselves to the case of ultrametric spaces over a join-semilattice $V$, and consider first the case where $V$ is complete and completely meet-distributive.

A \emph{join-semilattice} is an ordered set in which two arbitrary elements $x$ and $y$  have  a join, denoted by $x\vee y$,  defined as the least element of the set of common upper bounds of $x$ and $y$.

Let $V$ be  a join-semilattice with a least element, denoted by $0$.
A \emph{pre-ultrametric space} over $V$ is a pair $\mathbf  E:=(E,d)$ where $d$ is a map from $E\times E$ into $V$ such that for all $x,y,z \in E$:
\begin{equation} \label{eq:ultra1} 
d(x,x)=0,\; d(x,y)=d(y,x)  \text{~and } d(x,y)\leq d(x,z)\vee d(z, y).
\end{equation}

\noindent The map $d$ is an \emph{ultrametric distance} over $V$ and $\mathbf E$ is an \emph{ultrametric space} over $V$ if $\mathbf E$ is a pre-ultrametric space and $d$ satisfies \emph{the separation axiom}:
\begin{equation} \label{eq:ultra2}
d(x,y)=0\; \text{implies} \;  x=y .
\end{equation}

Any family   $\mathbf R:=(E, (\rho_i)_{i\in
I})$ of equivalence relations on  a set $E$ can be viewed as a pre-ultrametric space on $E$. Indeed, 
given a set $I$,  let $\powerset (I)$ be the power set of $I$. Then $\powerset (I)$, ordered by inclusion, is a join-semilattice (in fact a complete Boolean algebra) in which the join is the union, and  $0$  the empty set.
 \begin {proposition}\label{prop:ultra1}
Let $\mathbf R:=(E, (\rho_i)_{i\in
I})$ be a family  of equivalence relations. For $x,y\in E$, set $d_{\mathbf R}(x,y):=\{i\in I: (x,y)\not \in \rho_i\}$. Then   the pair $\mathbf E_{\mathbf R}:=(E, d_{\mathbf R})$ is a pre-ultrametric space over $\powerset (I)$.

\noindent Conversely, let $\mathbf E:=(E,d)$ a pre-ultrametric space over $\powerset (I)$. For every $i\in I$ set $\rho_i:=\{(x,y)\in E\times E: i\not \in d(x,y) \}$ and let  $\mathbf R:=(E, (\rho_i)_{i\in
I})$. Then each $\rho_i$ is an equivalence relation on $E$ and $d_{\mathbf R}=d$.

\noindent Furthermore,  $\mathbf E_{\mathbf R}$ is an ultrametric space if and only if $\bigcap_{i\in I} \rho_i= \Delta_E:=\{(x,x): x\in E\}.$   \end{proposition}

For a join-semilattice $V$  with a $0$ and for two pre-ultrametric spaces $\mathbf E:=(E,d)$ and $\mathbf E':=(E',d')$  over $V$,  a \emph{non-expansive mapping} (or contracting map) from $\mathbf E$ to $\mathbf E'$ is any map $f:E\rightarrow  E'$ such that for all $ x,y\in E$:
\begin{equation}
d'(f(x),f(y))\leq d(x,y). \end{equation}
Pre-ultrametric spaces with  their non-expansive mappings and systems of equivalence relations with their relational homomorphisms are two faces of the same coin.   Indeed: 

\begin{proposition} \label{prop:ultra2} Let $\mathbf  R:=(E, (\rho_i)_{i\in
I})$ and $\mathbf R':=(E', (\rho'_i)_{i\in
I})$ be two family  of equivalence relations. A map $f:E\rightarrow E'$ is a  homomorphism from $\mathbf  R$ into $\mathbf R'$ if and only if $f$ is a non-expansive mapping from $\mathbf E_{\mathbf R}$ into $\mathbf E_{\mathbf R'}$.
\end{proposition}
The  proof  is immediate and left to the reader.

\subsection{Hyperconvexity} 

Most of the results of this subsection are borrowed from \cite{pouzet-rosenberg}. 

Let $V$ be a join-semilattice with a least element $0$. Let  $d_{\vee}: V\times V \rightarrow V$ defined by $d_{\vee} (x,y)= x\vee y$ if $x\not =y$ and $d_{\vee} (x,y)=0$ if $x=y$. 

\begin{lemma} The map $d_{\vee} $ is a ultrametric distance over $V$ satisfying: 
\begin{equation}\label{eq:maximum}
d_{\vee}(0,x)= x 
\end{equation} 
for all $x\in V$. 

This is the largest ultrametric distance over $V$ satisfying (\ref{eq:maximum}).  

\end{lemma}

\begin{proof}Let $x,y,z$. If two of these elements are equal, the triangular ineqality holds. Otherwise we have trivially $d_{\vee}(x,y)= x\vee y\vee z= d_{\vee}(x,z) \vee d_{\vee}(z,y)$.  This proves that $d_{\vee}$ is an ultrametric distance. If $d$ is any ultrametric distance satisfying (\ref{eq:maximum}) then $d(x,y)\leq d(x,0) \vee d(0,y)=x\vee y$ for every $x,y\in V$. If $x\not =y$ we get $d(x,y)\leq d_{\vee}(x,y)$ and if $x=y$ we get $d(x,y)=0=d_{\vee}(x,y)$. 
\end{proof}

Let  $x,y$ be  two elements of $V$.

If $d$ is any ultrametric distance over $V$  we have:

\begin{equation}
x\leq y\vee d(x,y)
\end{equation}
and 

\begin{equation}
y\leq x \vee d(x,y). 
\end{equation}

Let $D(x,y):= \{z\in V: x\leq y\vee z\}$. If $V$ is a distributive lattice then $D(x,y)$ is a filter. 
Indeed, if $x\leq y\vee  z_1$ and $x\leq y \vee z_2$, then $x\leq  (y\vee z_1) \wedge (y \vee z_2)= y \vee (z_1\wedge z_2)$. 
Hence, if  $V$ is finite then $D(x,y)$ has  a least element, the \emph{residual}  of $x$ and $y$. 

In full generality,  the \emph{residual} of two elements $x,y$ of  a join-semilattice $V$ (or even a poset) is  the least element $x\setminus y$, if it exists,  of the set $D(x,y)$. If $V$ is a Boolean algebra, this is the ordinary  \emph{difference} of $x$ and $y$. We say that $V$ is \emph{residuated} if the residual of any two elements exists.

\begin{lemma}\label{residual} Let $V$ be  a join semilattice with a least element $0$. 
If $V$ is residuated, then the map $d_V: V\times V\rightarrow V$ defined by $d_V(x,y):= (x\setminus y) \vee (y\setminus x)$ is  an ultrametric distance over $V$, and in fact the least possible distance $d$ satisfying condition (\ref{eq:maximum}).

\end{lemma}
\begin{proof}
 Clearly, $d_V(x,y)=0$ iff $x=y$ and $d_V(x,y)=d_V(y,x)$. 
Let $x\in V$. We have $0\setminus x=0$ and $x\setminus 0=x$ hence $d_v(0,x)=x$. 
Let $x,y\in L$. Clearly, $x\leq y\vee (x\setminus y)\leq d_V(x,y)$. Furthermore, $d_V(x,y)\leq x\vee y$. Hence the triangular inequality holds for $\{0, x,y\}$. Now, let $z\in V$. We have:

\begin{equation}\label{eq:residuation}
x\setminus y\leq (x\setminus z) \vee (z\setminus y).
\end{equation}
 Indeed, this inequality amounts to   $x\leq y\vee ( (x\setminus z) \vee (z\setminus y))$. An inequality which follows from the inequalities  $x\leq z\vee (x\setminus z)$ and $z\leq y \vee (z\setminus y)$.   
 
 The triangular inequality follows easily. 
\end{proof}

We say that  a complete lattice $V$ is \emph{$\kappa$-meet-distributive} if for every subset $Z\subseteq V$ with  $\vert Z\vert \leq \kappa$ and $y\in V$,
$$ \wedge  \{y\vee z: z\in Z\}= y\vee \bigwedge Z.$$ It is \emph{completely meet-distributive} if it is $\vert V\vert$-meet-distributive (beware, this terminology has other meanings). A meet-distributive lattice is also called a op-frame. This is a Heyting algebra w.r.t. to the join as the binary operation, and the involution equal to the identity. 
We have:

\begin{lemma}\label{residual3} Let $V$ be complete lattice. Then   $V$ is residuated if and only if it  completely meet-distributive. \end{lemma}

\begin{proof}
Suppose that $V$ is residuated. Let $y\in V$ and $Z\subseteq V$. Let  $x:= \wedge  \{y\vee z: z\in Z\}$ and let $x\setminus y$ be the residual of $x$ and $y$. 
Trivially $y\vee \bigwedge Z$  is a lower bound of $\{y\vee z: z\in Z\}$. Hence, $y\vee \bigwedge Z\leq  \wedge  \{y\vee z: z\in Z\}=x$. We claim that conversely $x\leq    y\vee \bigwedge Z$. It will follows that 
$\bigwedge \{y\vee z: z\in Z\}= y\vee \bigwedge Z$ as required. Indeed, from the fact that $x$ is a lower bound of $\{y\vee z: z\in Z\}$ we get that $x\setminus y$ is a lower bound of $Z$ and thus $x \setminus y \leq \bigwedge Z$. It follows that $x\leq y \vee x\setminus y\leq y\vee \bigwedge Z$, proving our claim.

 Suppose that  $V$ is complete and completely meet-distributive. Let $x,y\in V$ and $Z:= D(x,y)$. Since $V$ is complete, $\bigwedge Z$ exists. Due to complete meet-distributivity, we have $y\vee \bigwedge Z=\bigwedge \{y\vee z: z\in Z\}\geq x$, hence $\bigwedge Z$ is the least element of $Z$, proving that this is $x\setminus y$.  
\end{proof}

A residuated lattice does not need to be complete. For an example, if $V$ is a Boolean algebra, then $V$ is residuated and the distance over $V$,   $d_V(a,b)$ is equal to $a \Delta b$,  the symmetric difference of $a$ and $b$. There is a huge litterature about Boolean algebra viewed as metric spaces (e.g. \cite{blumenthal, blumenthal1, blumenthal-menger}). However, from Lemma \ref{residual3}, we have:

\begin{corollary}\label{cor:distributive1} A finite lattice is residuated iff it is distributive. 
\end{corollary}

From Lemma \ref{residual},  \ref{residual2} and Proposition \ref{prop-embedding},  we have \begin{theorem}
If a join-semilattice $V$ is completely meet-distributive  then 
it can be endowed with an ultrametric  distance $d_{ V}$ for which it becomes hyperconvex. Futhermore, every ultrametric metric space over $  V$ embeds isometrically into a power of $ V$. 
\end{theorem}

Rewriting  Proposition \ref {equivalenceAR} we obtain that   \emph{in the category of ultrametric spaces over  a completely meet-distributive lattice  $ V$, injective, absolute retracts,  hyperconvex spaces, spaces with the extension property and retracts of powers of $({ V},d_{ V})$ coincide}. Also, with Proposition \ref {prop:injectiveenvelope}, . 
Every ultametric space has an injective envelope.

This  result obtained for general Heyting algebras in \cite{jawhari-al}  has been independently obtained by \cite{bayod-martinez} in 1987 for ordinary ultrametric spaces, see  \cite{ackerman} for generalizations. 


\subsection{Preservation}

Let $\mathbf E:=(E,d)$ be a metric space over a join-semilattice $V$. For each $r\in V$ set $\equiv_r:= \{(x,y)\in E: d(x,y)\leq r\}$. Let $\Eqv_{d}(E):= \{\equiv_r: r\in V\}$.  Let $\mathbf F:=Hom(\mathbf E, \mathbf E)$ be the set of non-expansive  maps from $\mathbf E$ into itself, let  $\mathcal E_{\mathbf F}:= (E, \mathbf F)$ be the algebra made of  unary operations $f\in \mathbf F$ and let  $\Cong_{d}(E):= \Cong(\mathcal E_{\mathbf F})$ be the set of congruences of this algebra, that is  the set of all equivalence relations on $E$ preserved by all contractions from $\mathbf E$ into itself.

\begin{proposition}\label{fourre-tout} Let $\mathbf E:= (E,d)$ be an ultrametric space over a join-semilattice $V$ with a least element $0$. Then:
\begin{enumerate}

\item $\Eqv_{d}(E) \subseteq \Cong_{d}(E)$. 

\item A map $f:E \rightarrow E$ is a non-expansive map of $\mathbf (E, d)$ into itself iff  it preserves all members of $\Eqv_{d}(E)$. 

\item If the meet of every non-empty subset of $V$ exists, then $\Eqv_d(E)$ is an intersection closed subset of $\Eqv(E)$,  the set of equivalence relations on $E$. 
\item The set $\Cong_{d}(E)$ is an algebraic lattice; furthermore, for every $(x,y) \in E \times E$,  the  least member   $\delta(x,y)$ of $\Cong_{d}(E)$ containing $(x,y)$ is a compact element of $\Cong_{d}(E)$ and   $\delta(x,y)$ is included into $\equiv_r$, where   $r:= d(x,y)$. 

\item Any two members of $\Eqv_d(E)$ commute and $\equiv_r\circ \equiv_s=\equiv_s\circ \equiv_r= \equiv_{r\vee s}$ for every $r,s\in  V$  iff $\mathbf E$ is convex. 
\item $\mathbf E$ is hyperconvex iff $\Eqv_d(E)$ is a completely meet-distributive sublattice of the lattice $\Eqv(E)$ of equivalence relations  on $E$. 
\end{enumerate}
\end{proposition}
\begin{proof}The first two item are  immediate. Trivially,  each $\equiv_r$ is an equivalence relation and it  is preserved by all contracting maps.  Item (3). Let $\equiv_{r_i}$, $i\in I$ be a family of members of  $\Eqv_d(E)$ then 
$\bigcap_{i\in I} \equiv_{r_i}$ equals $\equiv_r$ where $r:= \bigwedge \{r_i: i\in I\}$. Item (4). Since $(x,y)\in \equiv_r$ and $\equiv_r$ is preserved by all contractions, we have $\delta(x,y)\subseteq \equiv_r$. Since $\Cong_{d}(E)$ is the  congruence lattice of an algebra it is algebraic. The fact that $\rho(x,y)$ is algebraic follows from the algebraicity of $\Cong_{d}(E)$. 
 Item (5) is  Proposition 3.6.7 of \cite {pouzet-rosenberg}. We recall the proof.  Let $r,s\in V$. Due to the triangular inequality, we have $\equiv_s\circ \equiv r \subseteq \equiv _{r\vee s}$. We claim that the equality holds whenever  $\mathbf E$ is convex. Let $t:= r\vee s$ and  $(x,y)\in \equiv_t$. Since $d(x,y)\leq t= r\vee s$ and $\mathbf E$ is convex, the closed balls $B(x, r)$ and $B(y, s)$ intersect. If  $z$ belongs to this intersection, then $d(x,z)\leq r$ and $d(y, s)\leq  t$ hence $(x,y)\in \equiv_s\circ \equiv_r$. This proves our claim. Conversely, let $B(x, r)$ and $B(y, s)$ with $d(x,y)\leq r\vee s$, that is $(x,y)\in \equiv_{r\vee s}$. We have $\equiv_r \vee \equiv_ s \subseteq \equiv_{r \vee s}$ and since $r$ and $s$ commute, $\equiv_{s}\circ \equiv_r=  \equiv_r \vee \equiv_s$. Due to our assumption $\equiv_{r} \vee \equiv_{s}= \equiv_{r \vee s}$, hence $\equiv_r \vee \equiv_s$ is the join in $\Eqv_d(E)$;  furthermore, since  $\equiv_{s}\circ \equiv_r=  \equiv_r \vee \equiv_s$ there is some $z\in E$ such that $z\in  B(x, r)\cap B(y, s)$. 
Item (6) is Proposition 3.6.12 of \cite {pouzet-rosenberg}.
\end{proof}

\begin{corollary} If $\mathbf E:=(E,d)$ is convex the map $r\rightarrow \equiv_r$ is  a lattice homomorphism from $V$ into $\Eqv_d(E)$.  
\end{corollary}

\begin{theorem}\label{thm:hyperconvex} If an ultrametric space  $\mathbf E:=(E,d)$  is hyperconvex, then every member of  $\Cong_{d}(E)$ is a join of equivalence relations of the form $\equiv_r$, for $r\in V$. 
\end{theorem}

\begin{proof} Let $\rho$ be an equivalence relation on $E$. Let $(x,y)\in \rho$ and $r:= d(x,y)$. We claim that if $\rho$ is preserved by every contracting map then $\equiv_r\subseteq \rho$. Indeed, let $(x',y') \in \equiv_r$. The (partial) map $f$ sending $x$ to $x'$ and $y$ to $y'$ is contracting. Since $\mathbf E$ is hyperconvex, it extends to $E$ to a non-expansive map $\overline f$. Since $\rho$ must be preserved by $\overline f$, and $(x,y)\in \rho$,  we have $(x',y')\in \rho$. This proves our claim. From item (4) of Proposition \ref{fourre-tout} it follows that $\delta(x,y)= \equiv_r$. Also,  $\rho$ is the union of all $\equiv_r$ it contains. 
\end{proof}

\begin{lemma}\label{residual2}  If  $L$ is an algebraic lattice  then  the residual of two compact elements (provided it exists) is compact. 
\end{lemma}
 \begin{proof} Suppose $x$ and $y$ compact. Suppose $x\setminus y\leq \bigvee Z$ for some subset $Z$ of $L$.  We have $x \leq y \bigvee Z$. Since $x$ is compact, $x\leq y \bigvee Z'$ for some finite $Z'\subseteq Z$. Since $x\setminus y$ is the least $z$ such that $x\leq y\vee z$, we have $x\setminus y\leq \bigvee Z'$ proving that $x\setminus y$ is compact. 
\end{proof}

\begin{theorem} Let $L$ be an algebraic lattice and $K(L)$ be the join-semilattice of compact elements  of $L$. If $L$ is completely meet-distributive then $K(L)$ has an ultrametric structure 
and  $L$ is isomorphic to the set of equivalence relations on  $K(L)$  preserved by all contracting maps on $K(L)$. 
\end{theorem} 

\begin{proof} Due to Lemma \ref{residual} and \ref{residual2}, we may define on  $V:= K(L)$ the distance $d_V$. Due to meet-distributivity, $V$ is hyperconvex. According to Theorem \ref{thm:hyperconvex} each equivalence relation preserved by all contracting operation is a join of equivalence relations of the form $\equiv_r$ for some $r:=d_V(a,b)$. 
\end{proof}
\begin{corollary} If $V$ is a finite distributive lattice, then $V$ is isomorphic to the lattice of equivalence relations preserved by all contracting maps from $V$ into itself, $V$ being equipped with the distance $d_V$. 
\end{corollary} 

Hence, $V$ is representable as  the lattice of congruences of some algebra. In fact it is  strongly representable  \cite{quackenbush-wolk}. Dilworth  proved that it is representable as the lattice of congruences of some lattice  \cite{gratzer-schmidt}.  Define an \emph{arithmetic lattice} as a sublattice of the lattice $\Eqv(E)$ of equivalence relations on a set $E$ which is distributive and such that these equivalence commutes (see Section \ref{section:arithmetical}
below). Then from $(6)$ of Proposition \ref{fourre-tout} follows that \emph{every finite distributive lattice is representable as an arithmetical lattice}.%

\section{Arithmetical lattices}\label{section:arithmetical}

%
Let $\Eqv(E)$ be he lattice of equivalence relations  on a set $E$. A sublattice $L$ of $\Eqv(E)$ is \emph{arithmetical} (see \cite {pixley})  if it is distributive and pairs of members of $L$ commute with respect to composition, that is 

\begin{equation}
\rho\circ \theta=\theta \circ \rho\;  \text{for every}\; \theta, \rho\in L.
\end{equation}

This second condition amounts to the fact that the join $\theta\vee \rho$ of $\theta$ and $\rho$ in the lattice $L$ is their composition.

A basic example of arithmetic lattice is the lattice of congruences of $(\Z, +)$. The fact that pairs of congruences commute is easy (and interesting). If $\theta$ and $\rho$ are two congruences, take $(x,y)\in \rho\circ \theta$. Then, there is $z\in \Z$ such that $(x,z)\in \theta$ and $(z,y)\in \rho$. Let $r, t\in \N$ such that $\theta= \equiv_r$ and $\rho= \equiv_t$ , then there are $k,\ell\in \Z$ such that $z=x+k.r$ and $y=z+\ell.t$. Set $z':=x+\ell.t$ then $x\equiv_t z'\equiv_r y$ hence $(x, y)\in \equiv_r\circ \equiv_t= \theta\circ \rho$. Thus $\rho\circ \theta= \theta\circ \rho$ as claimed. 
 
As it is well known, if  $\theta$ and $\rho$ are two congruences,  $\theta= \equiv_{t}$ and $\rho=\equiv_{r}$ with $r, t\in \N$,  then  $\theta\vee \rho=\equiv_{ lcd \{t,r\}}$   
whereas, $\theta\wedge\rho= \equiv_{lcm\{t,r\}}$. Distributivity follows.

As it is well known (see \cite{pixley}), arithmetic lattices can be characterized in terms of the \emph{Chinese remainder condition}.

We say that a sublattice $L$ of $\Eqv(E)$ satisfies the \emph{Chinese remainder condition} if:

for each finite set of equivalence relations $\theta_1, \dots \theta_n$ belonging to $L$ and elements $a_1, \dots, a_n\in A$, the system:
\begin{equation}
x\equiv a_i (\theta_i), i=1, \dots, n
\end{equation} 

is solvable iff for all $1\leq i,j\leq n$

\begin{equation}
a_i\equiv a_j  (\theta_i\vee \theta_j).
\end{equation} 
Recall the following classical result:
\begin{theorem} A sublattice $L$ of $\Eqv(E)$ is arithmetical iff it satisfies  the Chinese remainder condition.
\end{theorem}

\subsection{Chinese remainder condition  and metric spaces} 

Chinese remainder condition can be viewed as a property of balls in a metric space. 
For an example,  in the case of $\Z$, if we may view the congruence class of $a_i$ modulo $r_i$ as the (closed) ball $B(a_i, r_i):= \{x\in E: d(a_i, x)\leq r_i\}$ in a metric space $(E, d)$, we are looking for an element of the intersection of these balls. As we have seen in Section \ref{section:generalized metric spaces}, conditions ensuring  that such element exists were  considered in metric spaces (generalized  or not),  Helly property and convexity being the keywords. 
In our case, we may observe that $\Z$ has a structure of ultrametric space, but  the set of values of the distance is not totally ordered. Ordering  $\N$ by the reverse of divisibility: $n\leq m$ if $n$ is a multiple of $m$, we get a (distributive) complete lattice, the least element being $0$, the largest $1$, the  join $n\vee m$ of $n$ and $m$ being the largest common divisor.   Replace the addition by the join and for two elements $a,b\in \Z$, set $d(a,b):= \vert a- b\vert$. Then $d(a,b)=0$ iff $a=b$; $d(a,b)=d(b,a)$ and $d(a,b)\leq d(a,c)\vee d(c,b)$ for all $a,b, c\in \Z$. With this definition, closed balls are congruence classes.  In an ordinary  metric space, a necessary condition for the non-emptiness of the intersection of two balls $B(a_i, r_i)$ and $B(a_j, r_j)$ is that the distance between centers is at most the sum of the radii, i.e. $d(a_i,a_j)\leq r_i+r_j$.  Here this yields $d(a_i,a_j)\leq r_i\vee r_j$ that is $a_i$ and $a_j$ are congruent modulo $lcd(r_i, r_j)$. When this condition suffices for the non-emptiness of the intersection of any family of balls they are said \emph{hyperconvex} and \emph{finitely hyperconvex} if it suffices for any finite family. Hence, Chinese remainder theorem of arithmetic is the finite hyperconvexity of $\Z$ viewed as an ultrametric space. 

This generalizes. 

Let $L$ be a sublattice of $\Eqv(E)$ that contains $0$ and is stable under the intersection of arbitrary meets. Let $\mathbf E:= (E,d)$ where $d: E\times E \rightarrow L$ is such that  $d(x,y)$ is the least member of $L$ containing $x$ and $y$. Then, trivially,  $\mathbf E$ is a generalized ultrametric space over $L$.

Naturally, we obtain: 
\begin{theorem}
$L$ is arithmetical iff and only if $\mathbf E$ is finitely hyperconvex.  
\end{theorem}

In Proposition \ref{equivalenceAR} was stated that hyperconvexity and one-extension property were equivalent provided that the set of values is Heyting. If it is not, a weakening is still valid.  Kaarli \cite{kaarli} obtained the following two results:
\begin{corollary}\label {oneextension}
If $L$ is arithmetical (and stable by arbitrary meets) then every partial function $f: B\rightarrow A$ where $B$ is a finite subset of $A$ which preserves all members of $L$ extends to any element $z$ of $A\setminus B$ to a function with the same property. 
\end{corollary}
We recall the proof.
\begin{proof}
Our aim is to find $x\in A$ such that for each $\theta \in L$ and $b\in B$, if $b\equiv z (\theta) $ then  $f(b)\equiv x (\theta)$. Let $B':= f(B)$. For each $b'\in B'$, let $\theta_{b'}$ be the least element of $L$ such that 
\begin{equation} 
b\equiv z(\theta_{b'})
\end{equation}
for all $b$ such that $f(b)= b'$. 

We claim that the system $x \equiv b'(\theta_{b'})$ is solvable and next that any solution yields the element we are looking for.  \end{proof}

\begin{corollary}\label{extension property}
If $L$ is arithmetical on a finite or  countable set $A$, then every partial function $f: B\rightarrow A$ where $B$ is a finite subset of $A$ which preserves all members of $L$ extends to a total function $\overline f$ with the same property. 
\end{corollary}
\begin{proof}Enumerate the elements of $A\setminus B$ in a list $z_0, \dots z_n\dots$. Set $B_n:= B\cup \{z_m:m<n\}$. Define $f_n: B_n\rightarrow A$ in such a way that $f_0= f$ and $f_{n+1}$ extends $f_{n}$ to the element $z_n$ and to no other. Set $\overline f:= \bigcup_n  f_n$. 
\end{proof}

We note that $\N$ ordered by reverse of divisibility is not meet-distributive. Still, it can be equipped with a distance (given by the absolute value). It is finitely hyperconvex, but it is not hyperconvex. Indeed, an infinite set of equations  does not need to have a solution while every  finite subset has one (for an example, let $a_{2n}:= 2$, $r_{2n}:= 2^{n}$, $a_{2n+1}:=3$, $r_{2n+1}:= 3^{n}$, then $d(a_{2n}, a_{2m})=0\leq r_{2n}\vee r_{2m}$, $d(a_{2n}, a_{2m+1})= 1\leq r_{2n}\vee r_{2m+1}= lcd(2^n, 3^m)=1$).

\subsection{Operations preserving all the congruences of $(\Z, +)$} 
Equip the set  $\Z$ of relative integers with the  operation $+$.  This algebra is a commutative group. As in every commutative group, a  congruence is determined by the class of $0$ (the others being translates). Such a  class is a subgroup. And this  subgroup is of the form $r.\Z$ for some non-negative $r$.  Hence, a congruence on $\Z$   is determined  by a non-negative integer $r$ and is defined by $x\equiv_r y$ if $x-y$ is a multiple of $r$.

  C\'egielski, Grigorieff and Guessarian (CGG), 2014 \cite{cgg1, cgg2} handled the description of (unary) maps preserving all congruences of $(\Z, +)$. Their  description is given in terms of Newton expansion.  If falls in the scope of the study of non expansive maps of an ultrametric space. Indeed, we may see  $\Z$ as an ultrametric space, values  of the distance being  the integers, ordered by multiplication, the least element being $0$ and the largest element $1$.

 The proof of CGC result  is by no means trivial. The first author found in  2016 a few lines proof of the main argument. It was presented with  a proof of GCC result in \cite{pouzet}. We will give it in Lemma \ref{wasdifficult} below.

Let $\mathcal C$ be the set of maps $f:\Z \rightarrow \Z$ which preserve all congruences on $\Z$. It is closed under product, hence it contains all polynomials with integer coefficients. But it contains others (e.g.  the polynomial  $g(x):= \frac{x^2(x-1)^2}{2}$  is a congruence preserving map on $\Z$). It is  \emph{locally closed}, meaning that  $f\in \mathcal C$  iff for every finite subset $A$ of $\Z$, (in fact, every $2$-element subset of $\Z$), the map $f$ coincides on $A$ with some $g \in  \mathcal C$ (in topological terms, $\mathcal C$ is a closed subset of the topological space  $\Z^{\Z}$ of maps $f:\Z \rightarrow \Z$ equipped with the pointwise convergence topology, the topology on $\Z$ being discrete).

Let $n$ be  a non-negative integer, let $lcm(n):= 1$ if $n=0$, otherwise  let $lcm ({n})$ be the least common multiple of $1, \dots n$, i.e. $lcm({n}):= lcm \{1, \dots, n\}$. If $X$ is an indeterminate (as well as a number) we set $X^{\underline 0}:= 1$, $X^{\underline 1}:=X$, $X^{\underline n}:= X\cdot (X-1)\cdot \dots \cdot (X-n+1)$. The \emph{binomial polynomial}  is  ${X \choose n}:= \frac{X^{\underline n}}{n!}$. 
%
%
%
%
%
%
%
%
%

CGG's result can be expressed as follows:
\begin{theorem}\label{CGGthm}
\begin{enumerate}[(1)]
\item Polynomial  functions of the form  ${lcm(n)}\cdot {x \choose n}$ preserve all congruences; 
 \item Every polynomial  function which preserves all congruences  is a finite linear sum with integer coefficients  of these polynomials; 
 \item The set $ \mathcal C$ of maps  $f: \Z\rightarrow \Z$ which preserve all congruences is the local closure of the set of polynomials preserving all congruences.
\end{enumerate} 
\end{theorem}

A more compact  form is given in $(b)$ of Lemma \ref{onepolynomialextension} below. Note that 
being closed  in  the set $\Z^{\Z}$ of all maps $f :\Z\rightarrow \Z$ endowed with the pointwise convergence topology, the set $\mathcal C$ is a Baire subset of $\Z^{\Z}$.  Hence, it is uncountable (apply Lemma  \ref{onepolynomialextension}, or observe that it has no isolated point and apply Baire theorem). In particular,  it contains   functions which are not polynomials. A striking example using Bessel functions is given in CGG's paper.    

%
We give the proof of Theorem \ref{CGGthm} below. Note first this:
\begin{lemma}\label{smalltrick}  Let $f(x):= {\lambda_k}\cdot {x\choose k}$. If $f$ preserves the congruences $\equiv_{i}$ for all $i:=0, \dots, k$ then $\lambda_k$ is a multiple of $lcm(k)$.
\end{lemma} 
\begin{proof}
 For $i:=0, 1, \dots, k-1$, we have $f(i)=0$. If $f$ preserves $\equiv_{k-i}$, $f(k)=f(k)-f(i)$ is a multiple of $k-i$, hence  
 $f(k)$ is a multiple of $k, k-1, \dots, 1$. Since $f(k)= \lambda_k$, the result follows. 
\end{proof}

Next,  we prove that $(1)$ of Theorem \ref {CGGthm} holds. 
\begin{lemma}\label{wasdifficult} Let $n$ be  a non-negative integer and  $f_n(x):= {lcm(n)}\cdot {x\choose n}$. Then $f$ preserves all congruences. \end{lemma} 

\begin{proof}

 This means that $f_n(x+k)-f_n(x)$ is divisible  by $k$ for every non-zero $k$. 
 
 This follows from the equalities: 
 
 \begin{equation} 
 {{x+k}\choose n}-{x \choose n}= \sum_{i=1, \dots, n}{x\choose {n-i} }\cdot{k\choose i}= \sum_{i=1, \dots n}{x\choose {n-i}} \frac{k}{i}{{k-1}\choose {i-1}}.
 \end{equation}
 \
 Indeed, $\frac{lcm(n)}{i}$ is an integer for every $i=1,\dots,  n$.  
To prove that the first  equality holds, it suffices to check that its holds for infinitely many values of $x$. So suppose $x, k\in \N$. In this case, the left hand side  counts  the number of $n$-element subsets $Z$ of a $x+ k$-element set union of  two disjoints set $X$ and $K$ of size $x$ and $k$, each $Z$ meeting $K$.  Dividing this collection of  subsets according to the size of their intersection with $K$ yields the right hand size of this equality. 
\end{proof}

We go to the proof of $(2)$ of Theorem \ref{CGGthm}.

We first recall the  description  of polynomial  functions with integer values given  by Polya in 1915 (cf Theorem 22 page 794 in Bhargava \cite{bhargava}).

\begin{lemma}\label{integervalues} Polynomial  functions from $\Z$ to $\Z$ are finite linear sums with integer coefficients of polynomial functions of the form ${x \choose k}$. 
\end{lemma}
\begin{proof}
Let $P$ be a polynomial of degree $n$ over the reals. Since the ${X\choose k}$, $k\in \N$,  have different degrees, they form a basis, hence 
$$P:= \lambda_0+ \dots+ \lambda_k \cdot  {X\choose k}+ \dots+ \lambda_n \cdot {X\choose n}$$ for some reals $\lambda_0, \dots, \lambda_n$. 

Since $[{X\choose k}](X=m)$ is a binomial coefficient (for $k\leq m$), every linear combination with integer coefficients of these polynomials takes integer values. Thus, if the $\lambda_k$'s are integers, $P$ takes integer values. Conversely, suppose  
 that the values of $P$ are integers  for $X:= 0, \dots, n$. A trivial recurrence on the degree  will show that the coefficients are integers. 
Indeed, let $$Q:= \lambda_0+ \dots+ \lambda_k \cdot  {X\choose k}+ \dots+ \lambda_{n-1}\cdot {X \choose n-1}.$$
Since  $Q(k)=P(k)$ for all $k\leq n-1$, each $Q(k)$ is an integer. Hence  induction applies to $Q$ and yields  that all $\lambda_0, \dots, \lambda_{n-1}$ are integers. Now, $P(n)= Q(n)+ 
\lambda_{n}\cdot  [{X\choose n}](X=n )$. Since $\lambda_0, \dots, \lambda_{n-1}$ are integers, $Q(n)$ is an integer; since  $[{X\choose n}](X= n)=1$, it follows that $\lambda_n$ is an integer.  This proves our affirmation about the integrality of the coefficients. 
\end{proof}

One can say a bit more:
\begin{lemma}\label{onepolynomialextension} $(a)$ Every map $f$ from a non-empty finite subset $A$ of $\Z$ and values in $\Z$ extends to a polynomial function with integer values and degree at  most $n$ where $n+1$ is the cardinality of the smallest interval containing $A$. $(b)$ For every map  $f:\Z :\rightarrow \Z$  there are integer coefficients $a_n, n\in \N$, such that  $$f(x)= \sum_{n=0, \infty} a_{n}\cdot P_{n}(x)$$ for every  $x\in \Z$, where $P_{n}$ the polynomial equal to 
 ${X+k\choose 2k}$ if $n=2k$  and equal to ${X+k\choose 2k+1}$ if $n= 2k+1$.
\end{lemma}

The proof is a bit tedious but not difficult, we leave it to the reader (see \cite{pouzet} for details).  Beware, Lagrange approximation will not do (e.g., in order to extend a map defined on a $2$-element subset, we may need a polynomial of large degree).

We  adapt the proof of Lemma \ref{integervalues} in order to prove $(2)$ of Theorem \ref{CGGthm}. 

\begin{lemma} \label{preservecongruence}Polynomial  functions from $\Z$ to $\Z$ which preserve all congruences are finite linear sums with integer coefficients of polynomial functions of the form $lcm(k)\cdot {x\choose k}$. 
\end{lemma}
\begin{proof}
 
Let $P$ be  a polynomial from $\Z$ to $\Z$. According to Lemma \ref{integervalues}
$$P:= \lambda_0+ \dots+ \lambda_k \cdot  {X\choose k}+ \dots+ \lambda_n \cdot {X\choose n}$$  where  $\lambda_0, \dots, \lambda_n$ are integers.
Suppose that $P(k)-P(k')$ is a multiple of $k-k'$ for all  $k, k':=1, \dots, n$. We prove by induction on the degree that $\lambda_k$ is a multiple of $lcm(k)$ for each $k:=1, \dots, n$. 
Let $$Q:= \lambda_0+ \dots+ \lambda_k \cdot  {X\choose k}+ \dots+ \lambda_{n-1}\cdot {X\choose {n-1}}.$$
We have  $Q(k)=P(k)$ for all $k\leq n-1$. Hence, $Q$ satisfies the property,  induction applies and yields  that all $\lambda_k$ are integer  multiples of $lcm(k)$ for $k\leq n-1$. Now, $P(n)= Q(n)+ 
\lambda_{n}\cdot  [{X\choose n}](X= n)$. Since $\lambda_k$ is a multiple of $lcm(k)$ for $k\leq n-1$, it follows from Lemma \ref{wasdifficult} that  $Q$ preserves all congruences, in particular $Q(n)-Q(k)$ is a multiple of $n-k$; since $P(n)-P(k)$ is a multiple of $n-k$,  $P(n)-Q(n)=\lambda_{n}\cdot  [{X\choose n}](X= n)= \lambda_{n}$ is a multiple of $n-k$. Hence $\lambda_n$ is a multiple of $1, \dots, n$. Proving that $\lambda_n$ is a multiple of $lcm({n})$. 
\end{proof}

The proof yields:

\begin{corollary}
If a polynomial of degree $n$ preserves all congruences of the form $\equiv_k$ for $k:=1, \dots n$, it preserves all congruences. 
\end{corollary}

\begin{lemma}\label{chinese} Every map $f$ from a finite subset $A$ of $\Z$ and values in $\Z$ which preserves the congruences extends to every $a\in \Z\setminus A$ to a map with the same property. 
\end{lemma}
\begin{proof}
This  follows from  the Chinese remainder theorem (see Corollary \ref {oneextension} in the previous subsection). 
\end{proof}
Lemma \ref{onepolynomialextension} becomes:
\begin{lemma} $(a)$ Every map $f$ from a finite subset $A$ of $\Z$ and  values in $\Z$ which preserves all congruences extends to a polynomial function preserving all congruences. $(b)$ Every map  $f:\Z :\rightarrow \Z$ which preserves all congruences is of the form  $$\sum_{n=0, \infty} a_{n}\cdot P_{n}$$ where each $a_n$ is an integer multiple of $lcm(n)$.
\end{lemma}
\begin{proof}
We  extend $A$ to a finite  interval $\overline A$. With  Lemma \ref{chinese}, we extend $f$ to $\overline  A$ to a map $\overline f$ which preserves all congruences. A proof as in  Lemma \ref {onepolynomialextension} will apply. We only need to check that $a_n$ is a multiple of $lcm (n)$ for each $n\in \N$. We do that by induction. We suppose $a_i$ is a multiple of $lcm (i)$ for each $i<n$. We need to prove that $a_n$ is a multiple of $lcm(n)$. The map  $\overline f_{\restriction A_{n+1}}$ preserves the congruences  $\equiv_1, \dots, \equiv_n$, hence, by the proof of Lemma \ref{smalltrick},   $a_n$ is a multiple of $lcm (n)$. 
\end{proof}
%

%
%

\section{Operations preserving  the congruences of other groups and monoids}
There are many results on the preservation of congruences of groups and monoids, this in relation with studies about polynomial completeness, see  \cite{kaarli-pixley}.  In this section we discuss the case of the group $\Z^n$. 
 
Let $n$ be a non-negative integer and $(\Z, +)^n$ the $n$-power of the additive group $\Z$.  Results about the preservation of congruences for $n\geq 2$  and for $n=1$ are   completely different. For $n\geq 2$ there are only countably many operations preserving the congruences: these operations are affine \cite{nobauer}, while for $n=1$ there are uncountably many as we have shown in the previous section. 

Suppose $n\geq 2$. let $k,l <n$, set  $\Z_{(k)}:=  \{(x_{i})_{i<n}:  x_i=0 \; \text{for all}\;  i\not = k\}$, set $\Z_{(k,l)}:= \{(x_{i})_{i<n}:  x_i=0 \; \text{for all} i\not  \in  \{k,l\}\;  \text{and}\;  x_k=-x_l\}$.

\begin{theorem}\label{thm:powerofZ}Let $n$ be a non-negative integer and  $f: 
(\Z, +)^n\rightarrow (\Z, +)^n$. The following properties are equivalent:
\begin{enumerate}[(i)]

\item $f$ preserves all congruences of $(\Z, +)^n$; 
\item $f$ preserves the congruences associated with the subgroups $\Z_{(k)}$ and $\Z_{(k,l)}$ for $k,l<n$;
\item There are $(a_i)_{i<n}\in  (\Z, +)^n$ and $m\in \Z$ such that 
$f((x_i)_{i<n})= (a_i)_{i<n}+ (m\cdot x_0, \dots, m\cdot x_{n-1})$ for all $(x_0, \dots, x_{n-1})\in (\Z, +)^n$. 
\end{enumerate}
\end{theorem}

The proof relies on the  properties of the three equivalence relations $\simeq_i$, for $i\leq2$, on  the square $A\times A$ of an abelian group $A$. These equivalences relations are:
$(x,y)\simeq_1 (x',y')$ if $x=x'$; $(x,y)\simeq_2 (x',y')$ if $y=y'$; $(x,y)\simeq_0 (x',y')$ if $x+y=x'+y'$. 

We recall the following result: %
%

\begin{theorem}(Theorem 3.2 of \cite{delhomme-pouzet})\label{thm:preserve}  A map $f:A\times A \rightarrow A\times A$ preserves $\simeq_i$ for $i=0,1,2$ iff $f(x,y):= (x_0,y_0)+ (h(x), h(y))$ for some $(x_0, y_0)\in A\times A$ and some  additive map $h$ on $A$ (i.e., satisfying  $h(x+y)=h(x)+h(y)$ for all $x,y\in A$). \end{theorem} 

\noindent {\bf Proof of Theorem \ref{thm:powerofZ}}
Implications $(i)\Rightarrow (ii)$ and $(iii)\Rightarrow (i)$ are obvious.  We prove $(ii)\Rightarrow (iii)$. Set $g((x_i)_{i<n}):= f((x_i)_{i<n})-f(0)$. Let $k,l<n$. The map $g$ preserves $\Z_{(k)}$, $\Z_{(l)}$ and $\Z_{(k,l)}$; setting $A:=\Z$ we may identify the direct sum $\Z^{(k)} \oplus \Z^{(l)}$ with $A\times A$ and the restriction of $g$ to the map $\overline {g_{k,l}}$ from $A\times A$ into itself associating to each pair $(x, y)$ the pair $(u,v )$ such that $u= g( (x_i)_{i<n})_k$ and $v= g( (x_i)_{i<n})_l$ for $(x_i)_{i<n}\in \Z^{(k)} \oplus \Z^{(l)}$ such that $x_k=x, x_l=y$. This map  preserves the three equivalences $\simeq_i$ above. Hence $\overline g_{k,l} (x,y)= (h_{k,l} (x) , h_{k,l} (y))$ where $h_{k,l}$ is an homomorphism of $\Z$. Hence, there is some  $m_{k,l}\in \Z$ such that $h_{k,l} (z)= m_{k,l}\cdot z$ for every $z\in A$. If $n=2$ this fact yields the result. If $n>2$ then $m_{k,l}$ is independent of $k$  and then of $l$. This also yields the result. 
\hfill $\Box$

Instead of preserving congruences of a group we may preserve the congruences of a monoid. For the free commutative monoid $\N^{n}$ the results are similar (see \cite{cgg3}).
We conclude by mentionning the following beautiful result of C\'egielski, Grigorieff and Guessarian about the free monoid. 

\begin{theorem} \cite{cgg4}Let $n\geq 3$. A map $g$  preserves all congruences on the free monoid $A^{*}$ with  at least three generators iff there exists $n\in \N$ and $a_0, \dots, a_{n-1} \in A^*$ such that  $g(x)= a_0x a_{1}\dots x a_{n-1}$. for every $x\in A^{*}$. 
\end{theorem}

No simple proof is known. The case of the free monoid with two generators is open.

\section{Rigidity and semirigidity}\label{section rigidity}

A relation,  and more generally a set $\mathcal R$ of relations,  on a set $E$ is \emph{rigid} if the only self-map $f$ of $E$ preserving this relation,  or this set of relations,  is the identity. It is \emph{strongly rigid} if the only operations   preserving $\mathcal R$ are the projections. If the relations are reflexive then they are also preserved by the constant maps. So we say that a relation or a set  $\mathcal R$ of relations is \emph{semirigid} if the only self-maps $f$ of $E$ preserving this relation,  or this set of relations,  is the identity  or a constant map. It is \emph{strongly semirigid} if the only operations   preserving $\mathcal R$ are the projections or the constant maps. We refer to  \cite{lan-pos,  benoit-claude, ivo, v-p-h} for sample  of results about these notions and to \cite{couceiro-haddad-pouzet-scholzel} for some generalizations.

A set  $\mathcal R$ of equivalence relations  on a set $E$  is strongly semirigid if and only if it is semirigid. Z\'adori in
1983 has shown in 1983 that there are  semirigid systems  of three equivalence relations on finite sets of size different from $2$ and $4$. 
We present a construction given in \cite{delhomme-pouzet}. It  leads to   the  examples given by Z\'adori and  to  many others and also extends to  some infinite cardinalities.


We recall the following lemma. 
\begin{lemma} (Pierce) If a set  $\mathcal R$ generates by means of joins and meets  the
 lattice of equivalences relations on $E$, $\vert E\vert\not =2$,  then $\mathcal R$ is semirigid.
\end{lemma}

The proof is simple, still it is illustrative. 
\begin{proof}For $x\not =y$, let $<x,y>$ be the equivalence relation such that $\{x,y\}$ is the only non singleton class. We prove that if $f$ preserves all the  $<x,y>$-equivalences then $f$ is the identity or  a constant. Indeed, suppose by way of contradiction that $f$ is neither the identity, nor a constant map. Not being the identity, there is some $x$ with $f(x)\not =x$. The elements $x$ and $f(x)$ are equivalent w.r.t. $<x,f(x)>$ hence $f(x)$ and $f(f(x))$ must be equivalent w.r.t. $<x,f(x)>$. 
This gives two cases
1)$f(f(x))=f(x)$ and 2) $f(f(x)=x$.  If $1)$ holds, then since $f$ is not constant, there is some $z$ such that $f(z)\not =f(x)$; but since $z$ and $x$ are equivalent w.r.t. $<z,x>$, $f(z)$ and $f(x)$ must be equivalent w.r.t. $<z,x>$. This implies $f(z)=f(x)$,  a contradiction. 

If $2)$ holds, pick any $z\not \in \{x,f(x)\}$. From the preservation of $<x,z>$,  get $f(z)=f(x)$; from the preservation of $<f(x), z>$ get $f(z)= x$, which is impossible. \end{proof}

According to Strietz \cite{Stri 77} 1977,  if $E$ is finite with at least  four elements,  four equivalences are needed to generate the lattice of equivalence relations on
$E$ by means of joins and meets.  In that respect, Z\'adori result proving that one can find  semirigid systems made of three equivalence relations is interesting.  His examples seem   a bit mysterious. Our construction shows that this is not the case.

\subsection{Zadori's examples}

Zadori's result reads as follows.

\begin{theorem}\label{zadori}  Let $A:= \{0, \dots, n-1\}$ with $n=3$
or $n>4$. The following system of three equivalence
relations $\rho,\sigma,\tau$  on $A$  is  semirigid.
\begin{eqnarray*}
\rho&=&\{\{0\},\{1,2,\ldots,k\},\{k+1,\ldots,2k+1\}\},\\
\sigma&=&\{\{0,1,k+1\},\{2,k+2\}\ldots,\{k,2k\}\},\\
\tau&=&\{\{1,k+2\},\{2,k+3\}\ldots,\{0,k,2k+1\}\}.
\end{eqnarray*}
\end{theorem}

One of Z\'adori's examples: case $n$ even.

\begin{figure}[h]
\includegraphics[width=7cm,height=8cm]{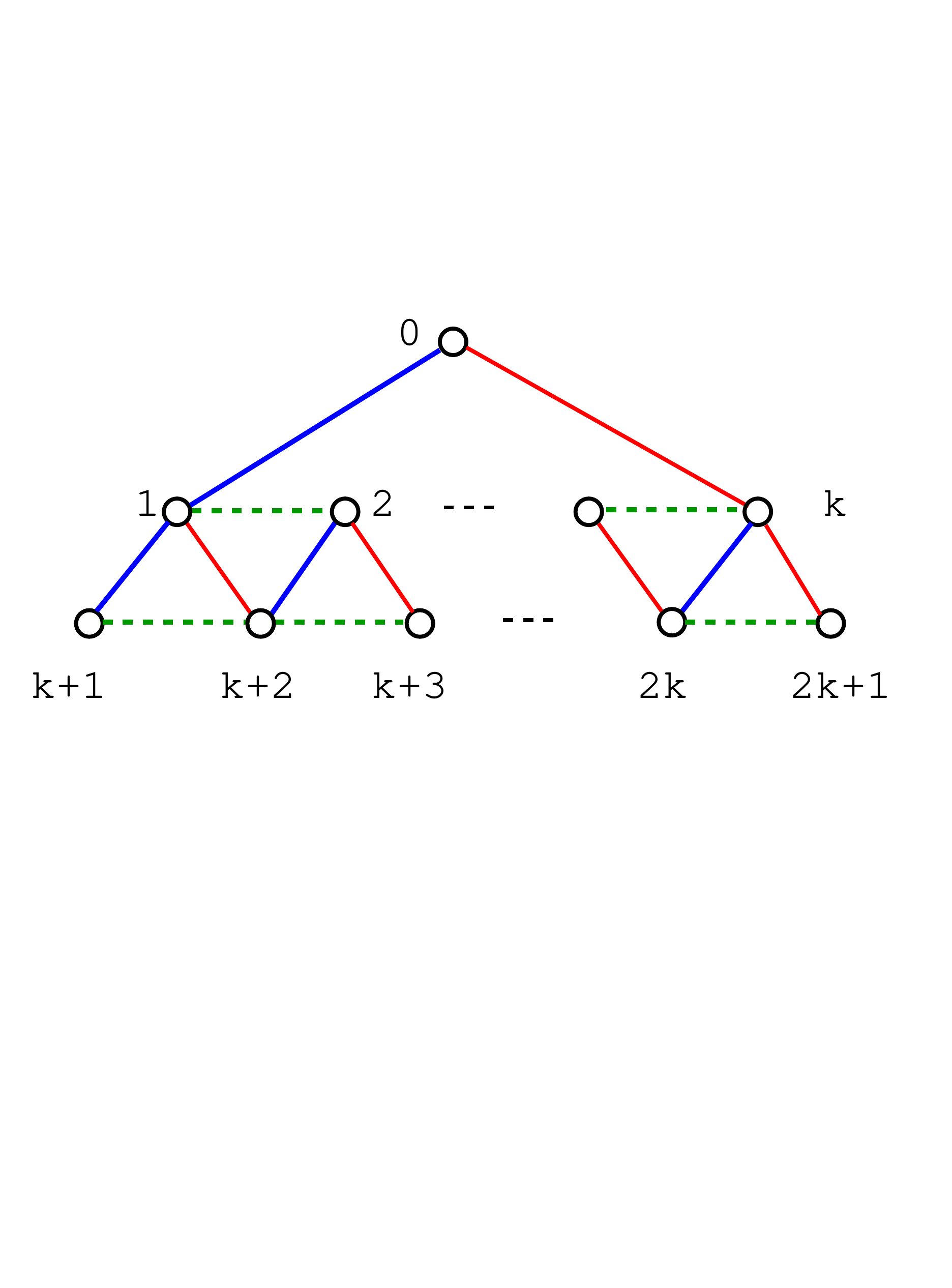}
\vspace{-6.0em}
\caption{Zadori'example: case $n$ even} \label{figure4}
\end{figure}

For $n$ odd, simply delete the node 0 which is located on top of the graph and relabel conveniently the vertices.

\subsection {A geometric construction} We describe  a general construction of semirigid systems of three equivalence relations  that  includes Z\'adori's examples

 Let $\R$ denotes  the real line,  let   $\R\times \R$ denotes the real
 plane. Let  $p_1$ and $p_2$ be the
      first and second projection from $\R\times\R$ onto $\R$ and  let 
      $p_0: \R\times \R \rightarrow \R$ the map $p_0:=p_1+p_2$. 
For $i= 0, 1, 2$, let  $\simeq_i$  be the kernel of $p_{i}$,
      i.e., for all $u, v \in \R\times \R$,  $u\simeq_iv$ if
      $p_i(u)=p_i(v)$.  Set $R:=(\R\times \R,  (\simeq_0,\simeq_1, \simeq_2))$.

The system $R$ is not semirigid. But, there are many subsets
 $C$ of the plane for which the system $R\restriction C$ 
is semirigid. As in \cite{delhomme-pouzet}, let us  introduce the notion of \emph{monogenic subset} of the plane.  

We define \emph{triangles} of the planes: there are the trivial ones: the singletons, the non-trivial ones are the $3$-element subsets $\{u_0, u_1,u_2\}$  of $E$ such that   $u_0\simeq_2 u_1$, $u_1\simeq_0 u_2$
      and $u_2\simeq_1 u_0$. The plane with the collection of triangles is an hypergraph and maps which preserve the three equivalences send triangles on triangles (possibly trivial). A subset $X$ of a subset $\mathcal C$ of the plane \emph{generates} $\mathcal C$ if for every (nontrivial) triangle $T$ included in $C$ there is a finite sequence of triangles $T_0, \dots T_n=T$ such that $\vert X\cap T_0\vert= \vert T_i \cap T_{i+1}\vert =2$  for $i<n$. 

We say that $C$ is \emph{monogenic}
if some subset of $C$ with at most two elements generates $C$.

Using Theorem 3.2 of \cite{delhomme-pouzet} recalled as Theorem \ref{thm:preserve}, one can prove. 

\begin{theorem}(Theorem 1.3  \cite{delhomme-pouzet})\label{maintheo} If  a finite subset $C$ of $\R\times \R$ is monogenic and has no center of symmetry  then $R\restriction C$ is semirigid.
\end{theorem}

The proof is based on Theorem 3.2 of \cite{delhomme-pouzet} recalled as Theorem \ref{thm:preserve}. 

A simple example, which is at the origin of this result, is the following.

For $n\in \N$ set $T_{n}:=\{(i,j) \in \N\times\N: i+j\leq
      n\}$. Then $T_{n}$ satisfies the hypotheses of Theorem
     \ref{maintheo}. 
Hence  $R\restriction T_{n}$ is semirigid. 
This yields an  example of  a semirigid system of three  pairwise
      isomorphic equivalence relations on a set  having
      $\frac{(n+1)(n+2)}{2}$ elements. 
      
\begin{figure}[h!]
\begin{center}
\includegraphics[scale=0.30]{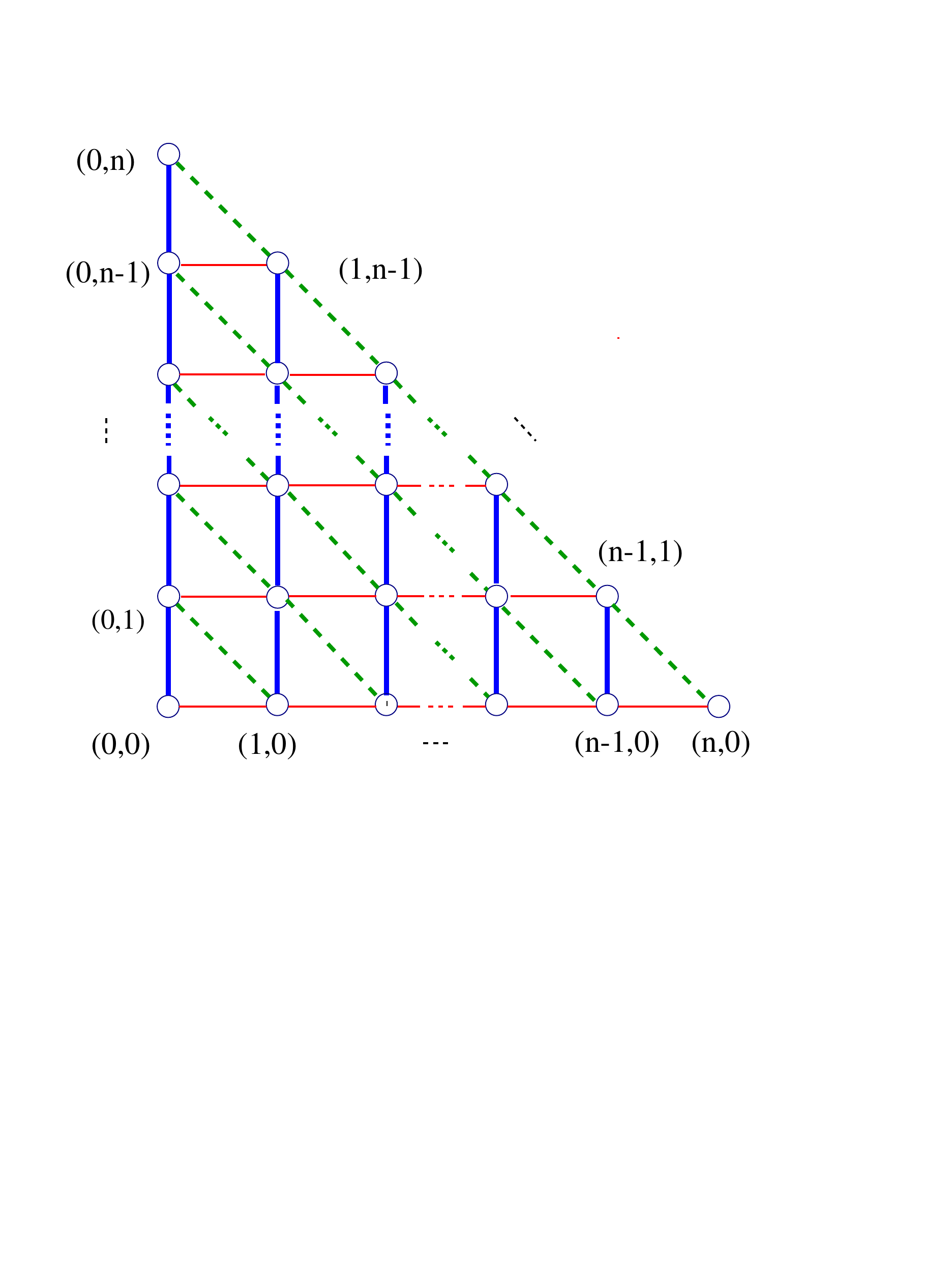}
\end{center}
\caption{The set $T_n$ with the three equivalence relations}
\end{figure}

Set  $T_{n,2}:=\{(i,j) \in T_{n}: i+j\in \{n-1, n\}\}$ and 
$T'_{n,2}:=T_{n,2}\cup\{(0,0)\}$. 

Both sets satisfy the hypotheses of Theorem \ref {maintheo}, hence
the induced $R\restriction T_{n, 2}$ and $R\restriction T'_{n, 2}$
      are semirigid. They are isomorphic to those of   Z\'adori.
The technique of the proof  of Theorem \ref{maintheo} yields:

\begin{theorem} \label{thm:infinite}For each cardinal $\kappa$, $\kappa\not \in \{2,4\}$ and  $\kappa\leq 2^{\aleph_0}$, there exists a semirigid system  of three equivalences on a set of cardinality $\kappa$.
\end{theorem}

For an example, let $A:=\Z$, $E:=A\times A$ and $$B:=\{(x,y)\in \Z\times \Z:  x+y\in \{1,2\}\}\cup \{(0,0)\}.$$ 
The system $\mathcal R\restriction B$ induced by the system $\mathcal R$ of three equivalence relations on $\R\times \R$  is semirigid.

\begin{figure}[h!]
\begin{center}
\includegraphics[scale=0.35]{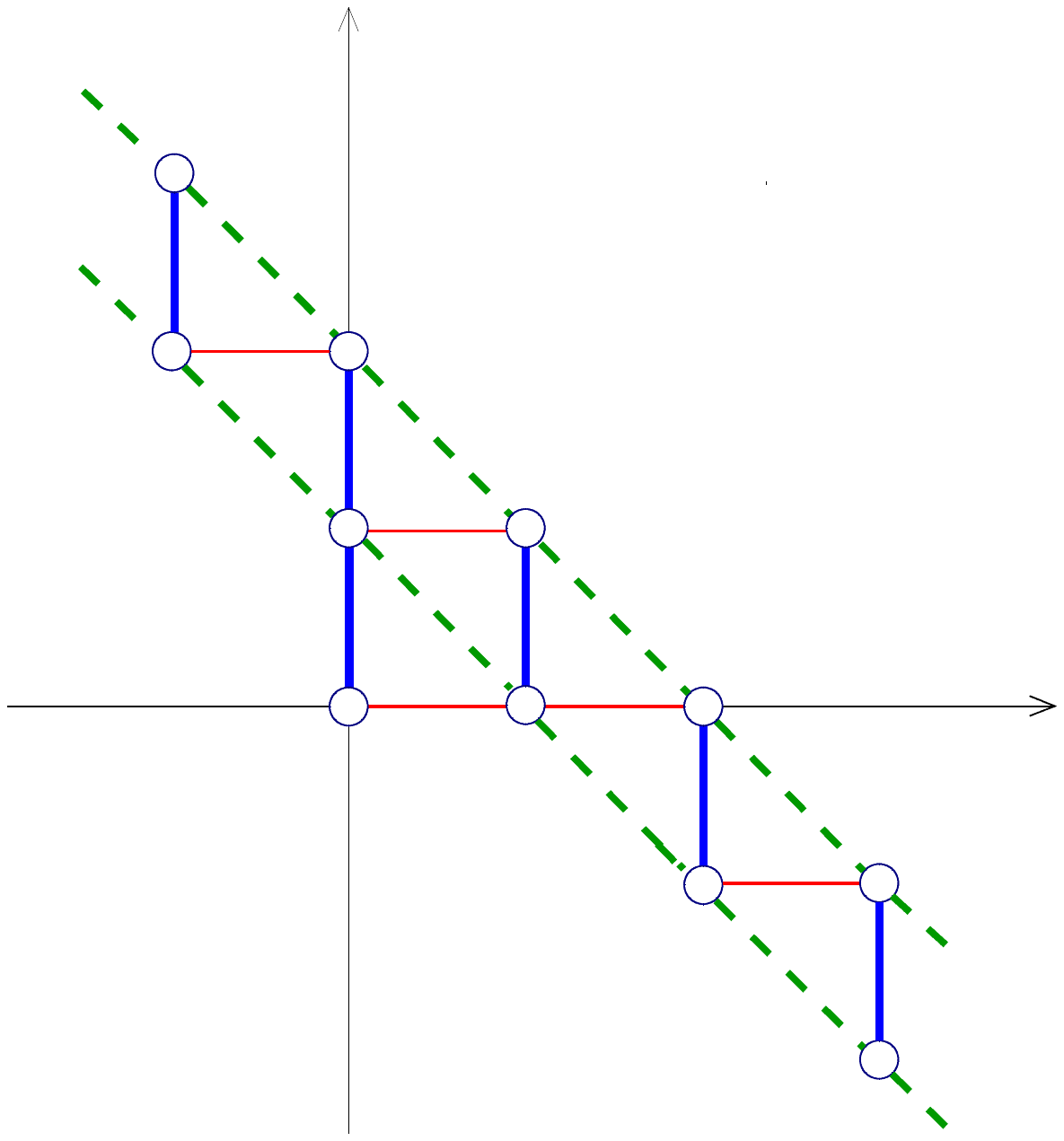}
\end{center}
\caption{An infinite band with a triangle}
\end{figure}
We recall that a subset $X$ of $\R$ is \emph{dense} if for every $x<y$ in $\R$ there is some $z\in X$ such that $x<z<y$. We also recall that every additive subgroup $D$ of $\R$ is either discrete, in which case $D=\Z\cdot r$ for some $r\in \R$, or dense. Let $D$ be an additive subgroup of $\R$ containing $\Z$. Set $\Delta:=\{(x,y)\in D\times D: 0\leq x, 0\leq y, x+y\leq 1\}$. If $D$ is a dense subgroup of $\R$ including $\Q$ then the system $\mathcal R\restriction ({B\cup\Delta})$ induced by the system $\mathcal R$ is semirigid.

\begin{figure}[h!]
\begin{center}
\includegraphics[width=3.5in]{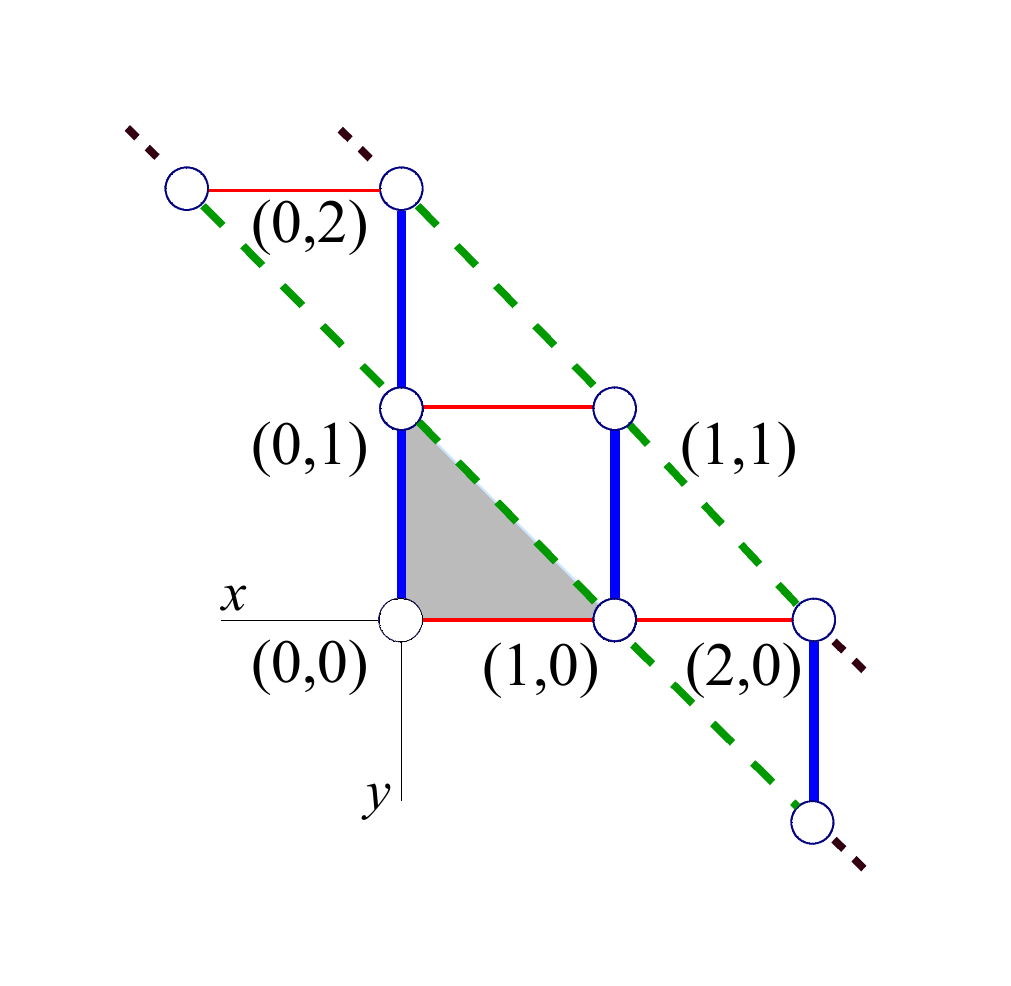}
\end{center}
\caption{A semirigid system  $B \cup \Delta$.}
\end{figure}\label{figuresemirigid}

Figure \ref{figuresemirigid} shows the set $B \cup \Delta$. The set $\Delta$ (shadowed) is
represented schematically and  the elements of the infinite set $B$ of two-dimensional integer points  by circles.

\begin{problem} Does the conclusion of Theorem \ref{thm:infinite} extend to every cardinal $\kappa>2^{\aleph_0}$?
\end{problem}

Related problems:
\begin{itemize}
\item  Describe the semirigid systems of the plane.
\item No useful characterization of semirigid systems is  known. 
\item No algorithm known to decide  in reasonable time whether or
      nor a system of $k$ equivalence relations on a set of size $n$ is
      semirigid or not.
\end{itemize}
\section{Conclusion}
With the present metric approach we have surveyed some categorical aspects of systems of reflexive binary relations.  We skipped the study of non necessarily reflexive relations and also the case of $n$-ary relations. This later case was considered in \cite{pouzet-rosenberg}. It could be studied in the same spirit as generalized metric spaces. Recent work on $n$-distances  by Marichal and his collaborators \cite{kiss-marichal-teheux} point in this direction.   

\end{document}